\documentclass[10pt]{article}
\usepackage{graphicx}
\usepackage{amssymb,amsthm,amsmath}

\newcommand{\eq}[1]{(\ref{#1})}

\newcommand{\mysgn}{\mathop{sgn}}

\renewcommand{\Re}{\mathop{\rm Re}\nolimits}
\renewcommand{\Im}{\mathop{\rm Im}\nolimits}
\newcommand{\la}{\lambda}
\newcommand{\tht}{\theta}
\newcommand{\om}{\omega}
\newcommand{\Om}{\Omega}
\newcommand{\eps}{\varepsilon}
\newcommand{\ksi}{\xi}

\newcommand{\Z}{\mathbb{Z}}
\newcommand{\R}{\mathbb{R}}

\newcommand{\CC}{\mathbb{C}}
\newcommand{\HHH}{{\cal{H}}}

\newcommand{\calI}{{\cal{I}}}

\newcommand{\rr}{\mbox{\boldmath$r$\unboldmath}}
\newcommand{\thtt}{\mbox{\boldmath$\theta$\unboldmath}}
\newcommand{\xx}{{\bf x}}

\newcommand{\Phibar}{\overline{\Phi}}
\newcommand{\vbar}{\overline{v}}

\newcommand{\del}{\partial}
\newcommand{\yp}{y_{+}^{(j+m)}}
\newcommand{\ym}{y_{-}^{(j-m)}}
\newcommand{\cto}{\rightarrow}

\newcommand{\laplace}{{\mathbf{\triangle}}}
\newcommand{\grad}{{\mathbf{\nabla}}}
\newcommand{\rme}{{\rm e}}
\newcommand{\rmi}{{\rm i}}
\newcommand{\inv}{^{-1}}
\newcommand{\sigdisc}{\sigma_{\rm disc}}
\newcommand{\sigess}{\sigma_{\rm ess}}
\newcommand{\nucl}[3]{%
  \ensuremath{%
    \phantom{\ensuremath{^{#1}_{#2}}}%
    \llap{\ensuremath{^{#1}}}%
    \llap{\ensuremath{_{\rule{0pt}{.75em}#2}}}%
    \mbox{#3}%
  }%
}
\newcommand{\gp}{Gross-Pitaevskii}
\newcommand{\schr}{Schr\"od\-inger}

\newtheorem{theorem}{Theorem}
\newtheorem{proposition}{Proposition}
\newtheorem{definition}{Definition}

\newtheorem{corollary}{Corollary}
\newtheorem{lemma}{Lemma}

\begin{document}

\title{Spectral stability of vortices in two-dimensional\\ 
Bose-Einstein condensates via the Evans function\\ 
and Krein signature}


\author{Richard Koll{\'a}r \\
Department of Applied Mathematics and Statistics\\
Faculty of Mathematics, Physics and Informatics \\
Comenius University \\
 Mlynsk{\'a} dolina, 842 48 Bratislava, Slovakia \\
E-mail: kollar\@ fmph.uniba.sk}

\date{Robert L.~Pego \\
Department of Mathematical Sciences \\
and Center for Nonlinear Analysis \\
Carnegie Mellon University\\  
Wean Hall 6113, 
Pittsburgh, PA 15213}



\maketitle

\begin{abstract} %
\noindent
We investigate spectral stability of vortex solutions of the
Gross-Pitaevskii equation, a mean-field approximation for
Bose-Einstein condensates (BEC) in an effectively two-dimensional
axisymmetric harmonic trap.  We study eigenvalues of the linearization
both rigorously and through computation of the Evans function, 
a sensitive and robust technique whose use we justify mathematically. 
The absence of unstable eigenvalues is justified {\em a posteriori}
through use of the Krein signature of purely imaginary eigenvalues,
which also can be used to significantly reduce computational effort. 
In particular, we prove general basic continuation results
on Krein signature for finite systems of eigenvalues in
infinite-dimensional problems.
\end{abstract}

\section{Introduction} \label{s:intro}%

{\em Background.} Since the experimental creation 
of Bose-Einstein Condensates (BEC) in alkali vapors in 1995 
\cite{cornell,ketterle}, BEC are one of the most active areas of modern 
condensed-matter physics.
A general overview of the subject can be found in
\cite{DGPS, Pethick}, and particularly in the review book \cite{KFC}.
In the Hartree-Fock mean-field approximation, BEC are modeled by the nonlinear
Schr{\"o}dinger equation (NLS) with a non-local nonlinearity. A
traditional simplification, replacing the non-local interaction
potential with a localized short-range interaction proportional to the
delta function, leads to the Gross-Pitaevskii equation
\begin{equation}\label{GP}
i \hbar \psi_t  = 
\left( - \frac {\hbar^2}{2M}  \laplace  +  V(\xx) + i\hbar \Om \del_{\theta} 
+ g |\psi|^2 \right) \psi \, ,
\end{equation}
where $\hbar$ is Planck's constant, $M$ is the atomic mass of atoms in
the condensate, 
$\theta$ is an azimuthal angle in cylindrical coordinates,  
and $g$ is an interaction strength parameter.
The total number of particles $N$ in the condensate is given by the integral
\begin{equation}
N = \int_{\R^3} |\psi|^2 dx^3\, ,
\end{equation}
and is conserved during the evolution of the system.
Equation \eq{GP} is a nonlinear Schr{\"o}dinger equation with a cubic
nonlinearity (focusing or defocusing depending on whether the
interaction is attractive or repulsive, respectively) and with a
spatially dependent trap potential $V(\xx)$ stationary in a frame rotating with frequency 
$\Om$ about the vertical axis. 
A rigorous mathematical 
justification of the Gross-Pitaevskii model for the BEC
ground state under various conditions directly
from many-body Schr{\"o}dinger equations was done in a series of papers of 
Lieb {\it et al.}~\cite{LS}--\cite{LS2}.

From the point of view of nonlinear waves, the interesting phenomena
is that the Gross-Pitaevskii equation, similarly to some other nonlinear 
Schr{\"o}dinger equations, supports the existence of various types of 
solitary wave solutions. In the two-dimensional setting we will study
in particular, there are vortex solutions which have the
form
$$
\psi(t,r,\theta) = e^{-i\mu t}e^{im\theta} w(r),
$$
where $r,\theta$ are polar coordinates, $m$ is vortex degree, $\mu$ is
the vortex rotation frequency (physically, chemical potential) 
and $w(r)$ is the radial vortex profile. 

Problems of stability for vortex solutions to various forms of nonlinear
Schr{\"o}d\-inger equations have drawn much attention in recent years. 
Related questions for models arising from nonlinear optics,
micromagnetics and Bose-Einstein condensation have been considered
extensively in the mathematical and physical literature.
For recent work concerning spectral stability questions for various
types of matter-waves, including vortices, vortex rings, multi-poles, soliton
and vortex necklaces in the presence of magnetic traps and optical
lattices, see for example \cite{KK, KKC, KKF, KFC, KKFM}. 
Rigorous mathematical results on these questions are rather few, 
however, due to the strong nonlinearity and complexity of the system.

{\em This work.}
In the present work we study the spectral stability of a single
two-dimensional axisymmetric vortex trapped in an axisymmetric harmonic trap.  
For this simple well-studied physical setting, we develop an approach 
that involves a combination of analytical and numerical tools 
which allows us to obtain reliable results for large particle number, 
well into the Thomas-Fermi regime (e.g. $N\sim10^6$ atoms of  \nucl{23}{}{Na}).
It is important to note that in the present 
axisymmetric setting, rotation of the trap does not influence the dynamic
stability of vortices, as the rotation term can be removed by
transforming to an appropriately rotating coordinate frame.

{\em Analytical results.}
Our study involves a number of analytical results that we prove
in sections 5 and 6 concerning the spectrum of the operator 
one obtains by linearizing about a vortex solution.  
In section 5, we prove that due to the harmonic trapping potential the 
essential spectrum of this operator is empty --- hence the spectrum
consists entirely of isolated eigenvalues of finite multiplicity.
The real part of any eigenvalue satisfies an explicit bound
depending only on the (non-dimensionalized) frequence $\mu$
and degree $m$ of the vortex, namely
\begin{equation}\label{e1.lam-bd}
|\Re \la| < 3(\mu  - m) \, .
\end{equation}
The eigenvalue problem breaks into an infinite system of coupled
pairs of ordinary differential equations (ODEs) for azimuthal
Fourier modes indexed by $j\in\Z$. As is known from \cite{GG},
only finitely many of these are relevant for possible instability,
namely the ones satisfying 
\begin{equation}\label{jbd0}
0<|j|<2m.
\end{equation}
We construct a globally analytic Evans function \cite{CKRTJ, Evans} 
associated with each of the ODE pairs, whose zeros correspond to the 
eigenvalues.  These results extend 
the approach of \cite{PW} for focusing-defocusing nonlinear Schr{\"o}dinger 
equations to handle a harmonic trapping potential.

One of the weaknesses in the numerical investigations in \cite{PW} was that 
one could not be confident that all unstable eigenvalues were detected,
due to the absence of any bound on their imaginary part.
No such bound is available in the present problem either. 
Nevertheless, we will show how one can indeed account for all 
possible instabilities through use of Krein signature (the sign of the
linearized energy of associated eigenmodes).

The key property of Krein signature that makes it useful is
its invariance under continuous variation of parameters such as
the standing-wave frequency $\mu$ and the size of the condensate. 
In section 6 we extend results that are well-known for finite-dimensional
systems to establish a general continuation property 
for any family of operators that is {\em resolvent-continuous},
a natural notion of weak continuity that one expects to hold
in many applications to infinite-dimensional systems.
Preservation of a definite signature is proved for
any finite system of imaginary eigenvalues for such a family.
By consequence, the only way eigenvalues can leave the imaginary axis
is through collision with eigenvalues of opposite Krein signature.

There are three reasons why Krein signature is extremely helpful and a
powerful tool in the present work.  First, it allows us to explain the
collisions (or avoided collisions) of eigenvalues found in our
numerical computations. Moreover, in combination with numerical plots,
it provides a numerically convincing {\it a posteriori} justification
that there are no unstable eigenvalues outside a certain fixed box in
the complex plane.  Finally, it can be used to significantly reduce 
the amount of necessary computation, as we shall discuss in section 7.2.

{\em Numerical methods.}
In order to locate zeros of the Evans functions using the argument
principle, we design and implement a numerical method somewhat more
robust than that used in \cite{PW}. We use a path-following technique and
multiple shooting to compute the nonlinear wave profile, and a rescaled
exterior-product representation of the Evans function to handle problems of
rapid growth and numerical dependence.
A numerical path-following technique for radially symmetric profiles
was also used by Edwards {\it et al.} \cite{EDCB}.

{\em Numerical results.}
In agreement with previous studies in the physical literature
\cite{GG, KKC, Pu, Skryabin}, we find a
singly-quantized vortex $(m=1)$ spectrally stable while the stability
of multiply-quantized vortices (with $m=2$ and 3) 
depends on the diluteness of the condensate,
with alternating intervals of stability and instability as $N g$ varies. 
Pu {\it et al.}  \cite{Pu} in their analysis use a different numerical method 
(finite elements) which is a priori less reliable and sensitive 
than the approach used here.  
Moreover, our results account for the appearance of {\em all} instabilities 
through the tracking of all eigenvalues of negative Krein signature. 
The presence of unstable eigenvalues (eigenvalues with positive real part) 
for certain parameters is also corroborated by direct simulation
of time-dependent dynamics based on the splitting scheme proposed in
\cite{Bao}. 

{\em Symmetries.}
In the computations one observes a special set of eigenvalues
which remain constant under variation of standing-wave frequency and
condensate size.  These eigenvalues arise from the 
symmetries of the Gross-Pitaevskii equation, as we discuss in an appendix. 
One symmetry is particularly interesting~---~a {\it breather boost},
which we found first described in \cite{PR}.
This is self-transformation of the Gross-Pitaevskii equation related to the
Talanov lens transformation, involving 
a time-periodic dilation of space with an appropriate radial phase
adjustment. This symmetry corresponds to eigenvalues 
$\pm2\rmi\omega$ of the \gp\
equation with the harmonic potential $V(\xx) = \frac12{\om^2}|\xx|^2$, 
linearized about a central vortex, with mode shapes
corresponding to infinitesimal breathing oscillations. 

{\em Related literature.}
Let us now discuss some known results on stability which are relevant to
our problem.  In general, two different concepts of stability are
distinguished in the literature: energetic and dynamic stability.  
A solution is energetically stable if it minimizes an associated energy
functional within a certain class of functions. 

The simplest energetic stability approach, where the minimization takes
into account only single vortex solutions with different charges,
indicates that a high enough trap rotation frequency can eventually
stabilize a vortex of any degree \cite{CD}.  On the other hand,
without external trap rotation, the energy of a single
multi-quantized vortex of charge $m$ is larger than the energy of $m$
singly-quantized vortices, and thus multi-quantum vortices are believed
to be unstable. The total energy in this case also depends on the
relative location of vortices as they tend to form regular hexagonal
arrays in harmonic traps. 

A mathematical framework for a rigorous variational approach was
discussed by Aftalion and Du \cite{AD}. Their method for effectively
2D condensates is parallel to the Ginzburg-Landau theory
of superconductors.  In \cite{VSS1} the authors claim that
sufficiently fast rotation in combination with a strong pinning
potential is capable of making even multi-quantum vortices energetically
stable.  

A detailed rigorous analysis was conducted by Seiringer \cite{S},
who studied regimes when a vortex solution can be a global energy
minimizer. He proves that for any $0 < \Om < \Om_c$ there exists
$m_{\Om}$ (independent of an interaction potential) such that all
vortices with  charge $m > m_{\Om}$ are energetically unstable; i.e.,
they are not global minimizers (ground states) of the energy
functional (see also \cite{Ignat}). Moreover, he proves that all
multi-quantized vortices, $m \ge 2$, become  energetically unstable
for a large enough value of the chemical potential of the condensate.
Finally, he proves that symmetry breaking of the axisymmetric
vortex solution is inevitable for any $m$, even for a singly-quantized
vortex for a large enough interaction strength, since no ground state
is an eigenfunction of the angular momentum.  The symmetry breaking of
a one-dimensional ground state is also demonstrated by a dynamical
systems analysis in \cite{JW} in the case of a double-well trapping
potential.

Energetic stability provides a sufficient condition for dynamic 
stability --- ground states of the energy functional 
are nonlinearly orbitally Lyapunov stable, i.e., if the initial data are
``close'' to the ground state solution then the perturbed solution remains
``close'' to the ground state solution for all times.
(See \cite{JW} for 
a detailed dynamic stability study of a one-dimensional model and
a sketch of the proof of well-posedness for the initial-value problem
for the Gross-Pitaevskii equation.)
As dynamic stability need not imply energetic stability, however,
it may not be possible to draw conclusions on dynamic instability 
directly from considerations of the energy functional. 
The study of linear or spectral stability 
can be helpful since one may detect possible instabilities 
due to the presence of eigenvalues in the right half-plane.

In the physical literature, 
Garc{\'\i}a-Ripoll and P{\'e}rez-Garc{\'\i}a \cite{GG} 
and Pu et al.~\cite{Pu} have studied linear stability of single multi-quantized
vortices using  equations equivalent to those here. 
The numerical results of \cite{Pu} agree substantially with those
of the present paper.
In \cite{GG} the search for instability is restricted only to the
so-called {\it anomalous modes} -- modes with a negative linearized energy.  
As pointed out in \cite{IM} for example, these modes are not intrinsically
unstable in the sense that some dissipation mechanism must be
introduced into the system for them to become relevant.  
The numerical techniques used in \cite{GG} and \cite{Pu} rely on
Galerkin-type approximations.
A finite-temperature generalization is further studied in \cite{VSS}.  

Use of the Krein signature as a tool to study stability of nonlinear waves has
appeared recently in various studies, see \cite{PC, HarKap, KapKrein, Sandstede, Pelinovsky} 
and references therein.
Skryabin in \cite{Skryabin2} (also see \cite{Skryabin}) studies a binary mixture
of trapped condensates using such information.
Kapitula {\it et al.} \cite{KKC} study
stability of various types of matter-waves including localized
vortices. Their perturbation argument, combined with topological
methods based on Krein signature, describes in detail transition to
instability in the limit of weak atomic interactions. 
Finally, Kapitula and Kevrekidis \cite{KaKe, KK}  have studied
Bose-Einstein condensates in the presence of a magnetic trap and
optical lattice, making efficient use of information on the Krein
signature of relevant eigenvalues. 

{\em Organization of this paper.}
First, Sections 2 and 3
introduce notation and recall background results regarding the
Gross-Pitaevskii equation and vortex solutions.  
Section 4 contains most of our analytical results, concerning
linearization, essential spectrum, bounds on eigenvalues, and
reduction to ODEs.  Moreover, we establish a precise asymptotic
description of eigenfunctions necessary for construction of
the Evans function. 
The Evans function itself is constructed in Section 5.  
In Section 6, we discuss the Krein signature.
We describe in detail our numerical procedure and discuss the numerical
results in Section 7.
Finally, in an Appendix we discuss symmetries and boosts of the
problem and relate them to the eigenvalues which do not change as
the standing wave frequency varies.


\section{\label{sec:level1-GP}%
The Gross-Pitaevskii Equation}
The behavior of low-temperature Bose-Einstein condensates (BEC) trapped
in the harmonic potential 
$V(\xx)$ rotating with the angular velocity $\Om$ about the $z$-axis
is well described by the time-dependent Gross-Pitaevskii equation.  
The wave function $\psi(\xx,t)$ 
satisfies \eq{GP} (in three dimensions)
with trapping potential $V(\xx)=V(x,y,z)$ given by
$$ 
V(\xx) = V_{3D}(\xx) =  \frac 12 M 
\left( \om_x^2 x^2 + \om_y^2 y^2 + \om_z^2 z^2 \right)\, . 
$$
The interaction strength parameter $g$ is
$$
g = g_{3D}  =  \frac {4\pi \hbar^2 a}{M} \, ,
$$
where $a$ is the $s$-wave scattering length \cite{LS}.
The term $\Om \del_{\theta}$ corresponds to
the angular momentum ${\mathbf \Om} \cdot ({\mathbf r} \times \grad)$ 
of the condensate caused by a rotating 
frame of coordinates. 
The total number of particles in the condensate $N$ is given by the integral
\begin{equation}
\int_{\R^3} |\psi|^2 dx^3= N\, ,
\end{equation}
and is conserved during the evolution of the system.
For disc-shaped (pancake) traps ($\om_z^2 \gg \om_x^2$, $ \om_z^2 \gg \om_y^2$) 
it was justified \cite{Bao} that  
the system is well approximated by a planar two-dimensional reduced model. 
The equation \eq{GP} formally does not change; one only needs to set
\begin{eqnarray*} 
V(x) & = & V_{2D}(x)  =  \frac 12 M \om_{tr}^2 \left( x^2 + \la_{tr}^2 y^2 \right) \, ,\\
g & = & g_{2D}  =  g_{3D} \cdot \left( \frac {M\om_z}{2\pi \hbar}\right)^{1/2} \, ,
\end{eqnarray*}
where $\om_x = \om_{tr}$ and $\om_y = \om_{tr} \la_{tr}$. 
For purpose of numerical investigations in this work 
the same values of parameters were used as in \cite{Pu}:
a condensate consisting of atoms of \nucl{23}{}{Na} is considered with $a = 2.75$~nm, 
$\om_z = 2 \pi \times 200$~Hz, 
$\om_{tr} = 2\pi \times 10$~Hz  ($\om_z \gg \om_{tr}$), $M =
10^{-26}$~kg and the Planck constant $\hbar = 6.6261 \times 10^{-34}$~Js.
A similar set of parameters used in the experiment as cited in \cite{PhysExp, VSS1}
with a \nucl{87}{}{Rb} condensate is: 
$M = 3.81 \times 10^{-26}$~kg, $a = 5.77$~nm, 
$\om = \om_z = \om_{tr} = 2\pi \times 200$~Hz. 
In these experiments the number of particles was approximately $N =
2\times 10^5$, horizontal and vertical condensate sizes were $R = 20$~$\mu$m,  $L = 10$~$\mu$m 
and the temperature $T_c = 1$~$\mu$K. Both \nucl{23}{}{Na} and \nucl{87}{}{Rb} represent alkali gases with 
a repulsive interaction potential. As an example of an attractive interaction, 
\nucl{7}{}{Li} with $a = -1.45$~nm can serve.
Note that the parameter $a$ can be tuned via Feshbach resonance.

In this work the following assumptions will be made. 
We assume that the magnetic trap is axisymmetric  ($\la_{tr} = 1$). 
Thus only two-dimensional wave functions
of the form $\psi = \psi(t,r, \theta)$ will be considered.
The symmetry of the trap also eliminates dependence of stability of axisymmetric vortices 
on the trap rotation. Hence, we will set ${\Om} = 0$. 
Moreover, although the model includes both attractive and repulsive
interaction interparticle potential (sign of the nonlinear
term), for simplicity  only the more interesting case of repulsive potential will be considered here
(some results concerning stability in transition between repulsive and attractive potential 
can be found in \cite{Pu}).

To nondimensionalize the equation \eq{GP} we use the following scaling
$$
t = (1/ \om_{tr}) \, t'\, , \qquad
x = \sqrt{{\hbar}\slash {M\om_{tr}}}\,  x'\, , \quad
\psi = \sqrt{\hbar \om_{tr}\slash {|g|}} \, \psi' \, .
$$
The Gross-Pitaevskii equation is then expressed (dropping the primes) as 
\begin{equation}
i \psi_t    =  -\frac 12 \laplace \psi + \frac12
r^2\psi+   |\psi |^2 \psi 
\label{GPfinal} 
\end{equation}
with 
\begin{equation}
\int_0^{\infty} \int_{0}^{2\pi}  |\psi|^2  r \, d\theta \, dr =   K = |g| NM/ \hbar^2\, .
\label{KK}
\end{equation}
The energy functional is given by 
\begin{equation}
E(\psi) =  \int_0^{\infty} \int_{0}^{2\pi}
\left[ \frac 12 |\grad \psi|^2
+\left( \frac 12 r^2 +  |\psi|^2 \right)|\psi|^2 \right] \, r d\theta \, dr \, .
\label{Escal1}
\end{equation}

Note that the Thomas-Fermi regime \cite{FS,LS} 
$Na / d_0 >> 1$ (here $d_0$ is the mean oscillator length,
$d_0 = \sqrt{\hbar/M \om_0}$ and $\om_0$ is the mean trap frequency 
$\om_0^3 = \om_x \om_y \om_z$) 
corresponds to $K \cto \infty$ since $K = 2|a|N \sqrt{2\pi M \om_z /
\hbar}$, i.e.,
\begin{equation}
K = \frac{N|a|}{d_0} 2 \sqrt{2\pi} \left(\frac{\om_{tr}}{\om_z} \right)^{1/3}\, .
\label{TFM}
\end{equation}
This is the limit under which Lieb and Seiringer \cite{LS} 
justified the Gross-Pitaevskii energy functional to be a good approximation for the $N$-body 
quantum system.
Note, that the only free parameter which stays in \eq{GPfinal}--\eq{KK} is 
$K$, the $L^2$-norm of the wave function $\psi$.


\section{\label{sec:level1-vortex}%
Vortex Solutions}
In this section we describe the structure of vortex solutions \cite{Neu}
to the Gross-Pitaevskii equation \eq{GPfinal},
of the form
\begin{equation}
\psi(t,r,\theta) = e^{-i\mu t}e^{im\theta} w(r),
\label{vortexS}
\end{equation}
and describe the numerical procedure 
which allows us to approximate them with high precision. 
Here $m$ is an integer, and in this paper it suffices 
to always assume that $m\ge1$ due to reflection symmetry. 

The radial profile function $w(r)$ of a vortex solution satisfies 
the equation
\begin{equation}
-w_{rr} - \frac 1r w_r + \frac{m^2}{r^2} w + r^2 w  
+ 2 |w|^2 w - 2\mu w = 0 \, , \quad r>0.
\label{GPradial}
\end{equation}
We will require that the radial profiles be non-negative and 
spatially localized, i.e., satisfy the boundary conditions
\begin{equation}
w(r) \ \mbox{is bounded as $r\cto 0^+$,}\quad
w(r) \cto 0^+  \ \mbox{as $r \cto \infty$.}
\label{BC}
\end{equation}
We note that the boundary condition as 
$r \cto 0^+$ implies $w(r) \cto 0^+$ for $m \ge 1$. 
 
{\em Existence.}
The existence of positive solutions to \eq{GPradial} 
for some $\mu$, corresponding to any given $K>0$ in \eq{KK},
can be proved using a well-known variational argument \cite{S}.
One minimizes \eq{Escal1} among functions with the spatial
dependence in \eq{vortexS}, with $m$ and $K$ fixed,
and a minimizer may be found that is positive.
For any positive solution of \eq{GPradial} with finite energy,
necessarily $\mu>m+1$, since multiplying \eq{GPradial} by $2\pi rw$
and integrating in $r$ yields
\begin{equation}
\mu K > \pi\int_0^\infty
\left(w_r^2 + 
\left( 
\frac{m^2}{r^2} +r^2\right)w^2\right) r\,dr \ge (m+1)K.
\label{Mestb}
\end{equation}
The last inequality follows since the integral
is minimized at a positive solution of 
\begin{equation}
-w_{rr} - \frac1r w_r + \frac{m^2}{r^2} w + r^2 w - 2\hat\mu w = 0 \, ,
\label{GPlin}
\end{equation}
which is \eq{GPradial} linearized at zero. Analysis of this equation
(see below) yields $\hat\mu=m+1$, $w=c w^{(m)}_0$ from \eq{e.Lag} below,
where $c$ is constant.
   
The following proposition describes a global bound on 
any vortex solution with positive profile.  
A proof can be found in \cite{thesis} (except for the statement that
$\mu>m+1$).  The bound \eq{Mestb} was also proved in \cite{S}. 

\begin{figure}[t]
\centering
\includegraphics[scale=0.6]{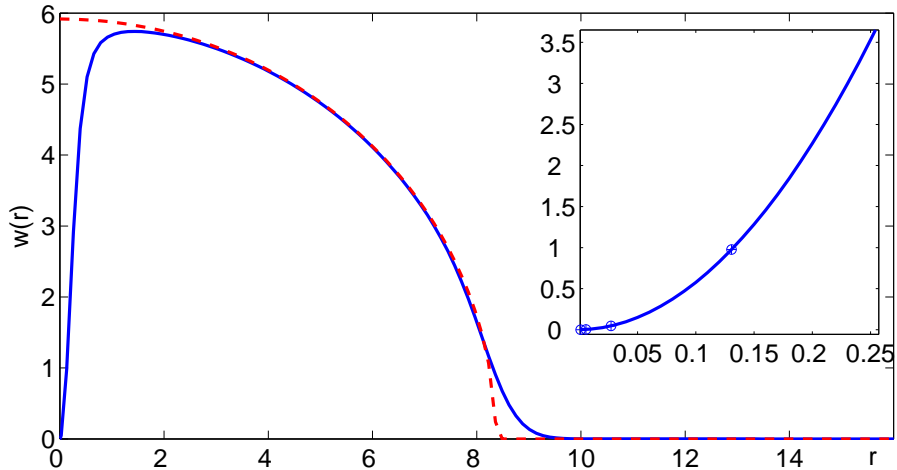}
\caption[The radial vortex profile, $m = 2$, $\mu \approx 35$]{
\label{fig:m2p35} 
The radial vortex profile of a multi-quantized vortex, $m = 2$, 
for the dimensionless parameter $\mu \approx 35$ corresponding to $N \approx 10^6$ 
particles of \nucl{23}{}{Na}. The quadratic profile in the Thomas-Fermi regime for 
the same parameters is also plotted (dashed line).
Detailed (quadratic) behavior close to the origin is on the inset.}
\end{figure}

\begin{proposition}
Let $w(r)$ be a finite-energy positive solution to \eq{GPradial},
satisfying \eq{BC}, where $m > 0$ is an integer. 
Then $\mu>m+1$, 
and $w(r)$ is increasing on $(0,R)$ and decreasing on $(R,\infty)$,
for some $R \in (m/ \sqrt{2\mu}, \sqrt{2\mu})$.
Moreover, for all $r>0$ we have
\begin{equation}
|w(r)|^2 < \mu - m .
\label{Mest}
\end{equation}
\label{profile_behavior}
\end{proposition}

The asymptotic behavior of a vortex profile can be determined directly
from \eq{GPradial} (see Fig.~\ref{fig:m2p35}).  
It is not difficult to see that as $r \cto 0^{+}$
the equation has the same character as the linear Schr{\"o}dinger
equation \eq{GPlin} (one argues as in \cite{IAWA}), and
$$
w(r) \sim d_0r^m \quad \mbox{for $m \ge 1$,}
$$
for some positive constant $d_0$. 
As $r \cto \infty$, the nonlinear term for a localized
solution becomes negligible and the linear equation \eq{GPlin}
is again a good approximation.
The proof that the positive solution $w(r)$ to \eq{GPradial} 
approaching~0 as $r \cto \infty$ satisfies 
$$
w(r) = O(r^{\mu-1} e^{-r^2/2})
$$ 
is also given in detail in \cite{thesis}.

The goal of this paper is to study spectral stability of solutions to
\eq{GPradial} both analytically and numerically.  
Naturally, for a careful numerical stability investigation 
it is crucial to obtain very precise
numerical solutions of \eq{GPradial} first.  
The approach used here is
based on  path-following along a branch bifurcating out of the trivial
solution $w = 0$ and is similar to the one used in \cite{EDCB}.  

{\em Bifurcation.} 
The bifurcation (and later stability) analysis
requires detailed information about the localized solutions to
\eq{GPlin}.  This equation has two independent general solutions ---
pro\-ducts of a polynomial,  a decaying Gaussian and a confluent
hypergeometric function.  The exact solutions (taking $\hat\mu=\mu$ henceforth) are
\begin{eqnarray*}
w_{(1)}(r) & = & r^m e^{-r^2/2} M\left( \frac{m+1-\mu}2, m+1, r^2\right)  \, ,
\\
w_{(2)}(r) & = & r^m e^{-r^2/2} U\left( \frac{m+1-\mu}2, m+1, r^2\right) \, .
\end{eqnarray*}
The confluent hypergeometric functions $M(a,b,x)$ and $U(a,b,x)$
are, in general, independent solutions to $x f'' + (b-x)f' -af = 0$ \cite{AS}.
Their asymptotics as $r \cto 0^+$ and as $r \cto \infty$ is, respectively,
\begin{eqnarray*}
w_{(1)}(r) &\sim& r^m \left(1 + O(r^2)\right) \, , \qquad
w_{(1)}(r) \sim  \displaystyle\frac{\Gamma(m+1)}{\Gamma(\frac{m+1-\mu}{2})}  
r^{-\mu - 1}e^{r^2/2} \left(1 + O(1/r^2)\right) , \\
w_{(2)}(r) &\sim& \displaystyle\frac{\Gamma(m)}{\Gamma(\frac{m+1-\mu}{2})} r^{-m} \left(1 + O(r^2)\right)\, ,\quad
w_{(2)}(r) \sim r^{\mu - 1} e^{-r^2/2} \left(1 + O(1/r^2)\right).
\end{eqnarray*}
The Wronskian of $w_{(1)}(r)$ and $w_{(2)}(r)$ is given by
\begin{equation}
W(w_{(1)}, w_{(2)}) = -\frac2r \frac{\Gamma (m+1)}
{\Gamma (\frac{m+1-\mu}2) }\, .
\label{Wrons}
\end{equation}
The only possibility for $w_{(1)}(r)$ (and similarly for $w_{(2)}(r)$) to
satisfy the boundary conditions at both ends is when the Wronskian \eq{Wrons} 
vanishes. This happens if $|\Gamma((m + 1-\mu)/2)| = \infty$, so 
$m+ 1-\mu = -2n$, $n$ a nonnegative integer.  
Therefore a non-trivial 
solution $w_n(r)$ to \eq{GPlin} approaching zero as 
$r \cto 0^+$ and as $r \cto \infty$ exists if and only if $\hat\mu = \mu_n$, where
\begin{equation} 
\mu_n  = m+1 +2n, \quad n=0,1,2,\ldots\, .
\label{spectrum}
\end{equation}

For $\mu$ given by \eq{spectrum} both solutions $w_{(1)}(r)$ and $w_{(2)}(r)$ reduce to
a constant multiple of the  single solution 
\begin{equation}\label{e.Lag}
w_n^{(m)}(r) = r^m e^{-r^2/2}L_n^{(m)}(r^2),
\end{equation}
where $L_n^{(m)}(r)$ 
is the generalized Laguerre polynomial, with $n$ the number of zeros of 
$L_n^{(m)}(r)$ for $r > 0$. The positive solution 
(the ground state of the associated energy functional) 
corresponds to $n = 0$, $\mu_0 = m+1$ and $w = w_0^{(m)}$.  

\begin{figure}[!tbp]
\centering
\includegraphics[scale=0.4]{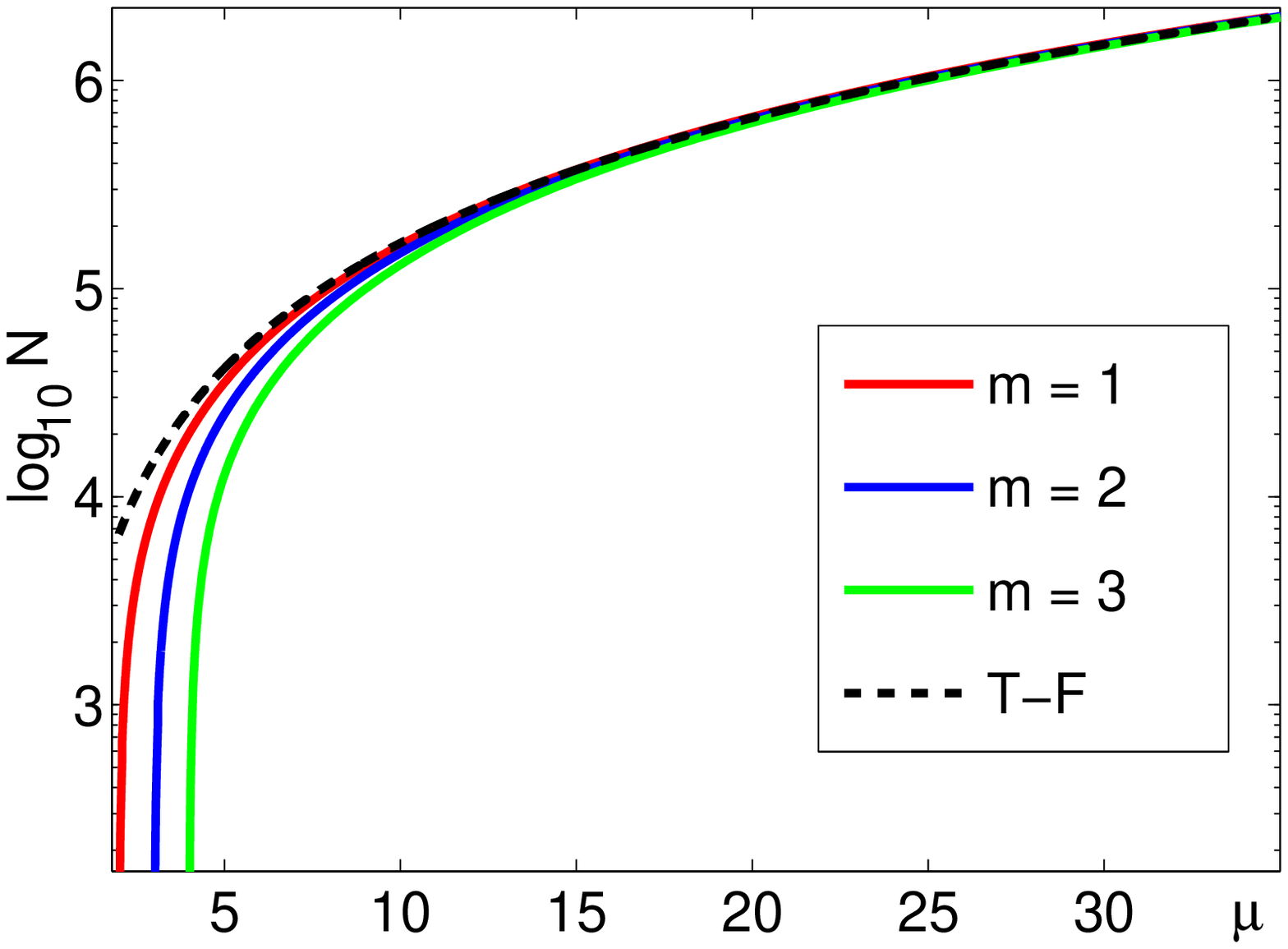}\qquad \qquad
\includegraphics[scale=0.4]{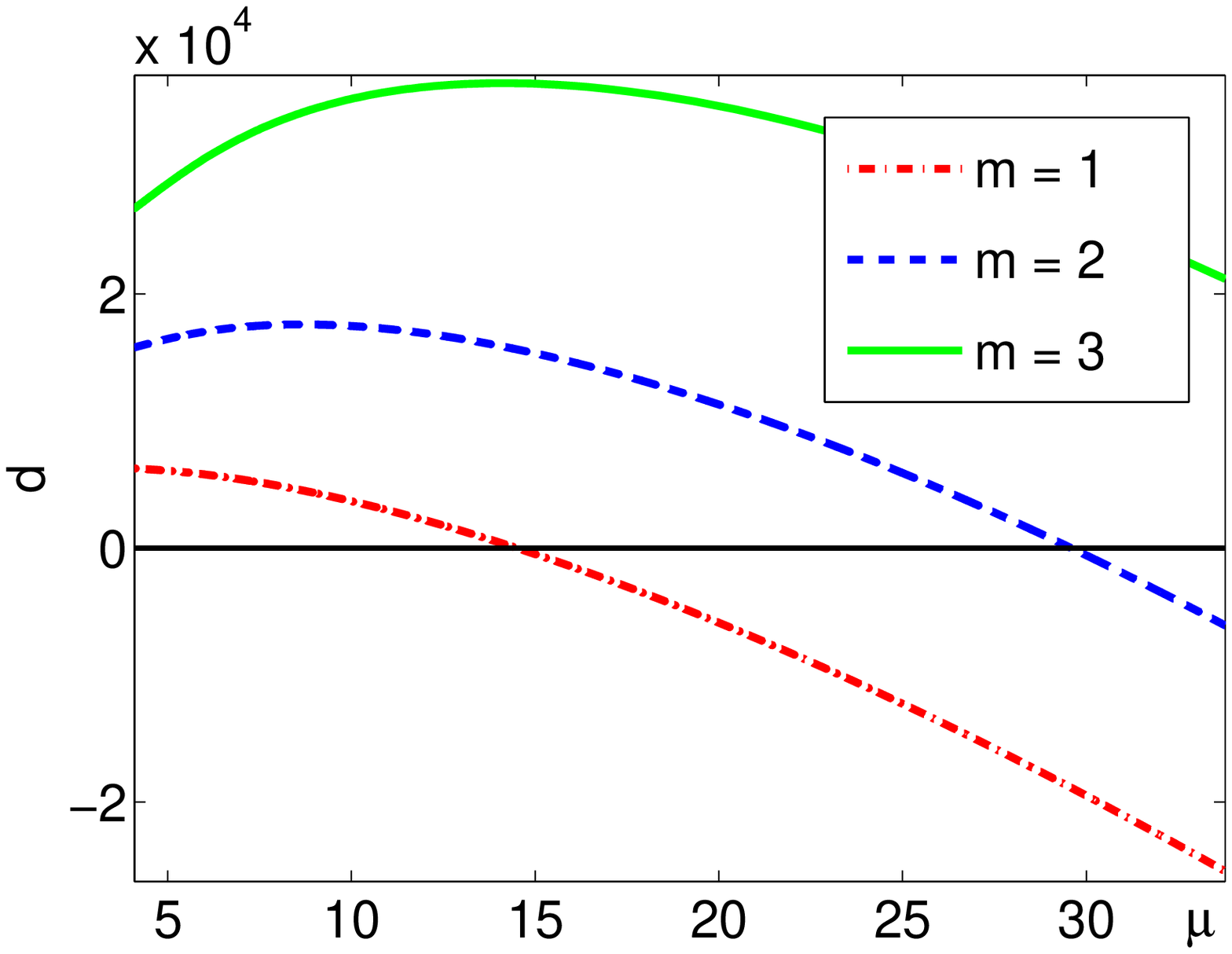}
\caption[Bifurcating branches of vortex solutions]{\label{fig:bif} 
Left panel: ($\log_{10} N$ vs. $\mu$) plot of branches of vortex solutions for $m =1, 2, 3$ 
(from left to right) emerging at  
$\mu_0 = m+1$ for \nucl{23}{}{Na} data.  Dashed curve represents the number of 
particles in the Thomas-Fermi parabolic regime 
$w_{TF}(r) = \sqrt{\mu - r^2/2}$. 
\newline
Right panel: The difference 
$d= \int |\Psi_{TF}|^2 - \left| \Psi \right|^2$ of the number of particles of 
vortex solutions  and the number of particles of $w_{TF}$
for $m = 1, 2, 3$.} 
\end{figure}

The numerical algorithm designed to find solutions to \eq{GPradial} is 
based on the following observation. 
It is reasonable to expect that introduction of the nonlinear 
term leads to the existence of a solution branch 
$(\mu(s), w_{\mu}(s))$ bifurcating from the trivial solution $w = 0$ for $\mu = \mu_0$. 
To justify such a behavior one can use the Crandall-Rabinowitz 
theorem \cite{CR,Nir}, to prove the following Theorem 
(for details see \cite{thesis}). 

\begin{theorem}
The solutions $(\mu,w)$ to \eq{GPradial} near $(\mu_0,0)$, $\mu_0 = m+1$, 
form a curve  
$$
s\mapsto (\mu(s), w(s)) = (\mu_0+\tau(s), sw_0^{(m)} + sz(s))\, ,
$$ 
where 
$s \mapsto (\tau(s), z(s)) \in \R \times \, 
\mbox{\rm span}\,  \left\{w_0^{(m)} \right\}^{\perp}$ is continuously differentiable 
near $s=0$, $\tau(0) = \tau'(0) = 0$, $z(0) = 0$ and 
\begin{eqnarray*}
&& w(s) \in X  =  \{ w :
e^{im\tht} w(r) \in \HHH\} \, ,
\\
&& \HHH  =  \{ u: u\in H^2(\R^2,\R^2), 
(x^2+y^2) u \in L^2(\R^2,\R^2) \} \, .
\end{eqnarray*}
\label{textbiftheorem}
\end{theorem}

Note that in addition to a branch of positive solutions $(\mu, w)$ 
bifurcating from the trivial solution at $\mu = \mu_0 = m+1$ in a direction $(0, w_0^{(m)})$ 
there are sign-changing branches of solutions $(\mu(s), w_n^{(m)}(s))$, $n = 1, 2, \dots$, 
bifurcating from the trivial solution at $\mu = \mu_n = m+1 + 2n$ in a direction $(0,w_n^{(m)})$. 
The proof is analogous.

{\em Numerics.} 
Hence it is possible to numerically trace the solution curve 
$(\mu,w)$ from the branching point $(\mu_0,0)$, see Fig.~\ref{fig:bif}(a).
The typical radial profile for $m=2$ is on Fig.~\ref{fig:m2p35}. 
Note that the Crandall-Rabinowitz bifurcation theory also provides
some linear stability information for solution branches given in
Theorem~\ref{textbiftheorem} \cite{CR, Nir}, but one must take the
stability results with caution. 
Although the theory predicts stability for the bifurcating branch, it
is only with respect to radially symmetric perturbations in
\eq{GPlin}. This is not sufficient to determine stability of vortex
solutions with respect to the full dynamics in \eq{GPfinal}. 

The behavior of norms relative to the norm of the parabolic Thomas-Fermi 
regime approximation is illustrated in Fig.~\ref{fig:bif}(b).
The first point on the approximate solution curve is set to be 
$(\mu_0, \eps w_0^{(m)})$,  $\eps=0.1$,
which is an $O(\eps^2)$ 
approximation of the exact solution. 
Then an implementation of a predictor-corrector path-following algorithm \cite{AG} 
is used to get solutions for large values of the parameter $\mu$. 

Since our stability study will require evaluations 
of the profile at any given point
within the computational domain, the precision of calculations 
is improved by optimizing the already calculated profile for any given $\mu$ 
by a multiple shooting procedure \cite{MSH}. 
This allows one to achieve high precision 
in evaluating $w(r)$ by simple integration from a nearby 
mesh point. Note that the calculation is almost
independent of the size of the parameter $\mu$ since 
the size of the computational domain, and so the number of necessary nodes, 
grows very slowly.  Therefore it is possible to reach large values of
$\mu$. Also note, that with the growing parameter $\mu$, the 
$L^2$-norm \eq{TFM} of profiles grows (Fig.~\ref{fig:bif})
and hence states far in the physically interesting Thomas-Fermi regime
for a wide range
of $\mu$'s are obtained for a small computational cost. 
On the other hand, as pointed in \cite{EDCB}, for computation for a single value of $\mu$ this
method has significant overhead. 


\section{Linearization and Reduction to Ordinary Differential Equations}
\label{sec:level1-stability}%

The goal of this section is to derive and study the linearization of 
\eq{GPradial} around the solutions constructed in the previous section
--- localized vortex profiles. 
The linearized equations have the same form as the so-called
Bogoliubov equations \cite{EDCB, Pethick} commonly used in physics
literature.  Note, however, that the relation between the derivation
here and the physical derivation of the Bogoliubov equations is
by no means straightforward. 

A small general perturbation of the vortex solution 
$\psi (t,r,\theta) = e^{i(-\mu t + m\theta)} w(r)$, 
where $w(r) = w_\mu(r)$ for a fixed parameter 
$\mu$, has the form 
$$
u(t,r,\theta) = e^{-i \mu t} \left( e^{im\theta} w(r) +
\eps v(t,r,\theta)\right)  \, .
$$
Neglecting nonlinear terms in \eq{GPfinal} yields 
\begin{equation}
i v_t = - \frac 12 \laplace v  - \mu v + 
\frac 12 r^2 v + 2 |w|^2 v +|w|^2 e^{2im \theta} \vbar \, .
\label{pertGP}
\end{equation}
The complex character of the equation \eq{pertGP} complicates the analysis. 
Therefore, we decompose the complex wave function $v$ as
$$
\Phi = 
\left(
\begin{matrix} 
\Phi_1 \\ 
\Phi_2 
\end{matrix}
\right)
= 
\left(
\begin{matrix}
\mbox{Re } v \\
\mbox{Im } v
\end{matrix}
\right)\, .
$$
The equation \eq{pertGP} is then equivalent to the real system
\begin{equation}
\del_t \Phi = 
A \Phi = 
J \left[
\frac 12 \laplace - \left( \frac 12 r^2 -\mu + 2|w|^2 
\right)
- |w|^2 e^{2m\theta J} R 
\right] \Phi \, ,
\label{linpGPvec}
\end{equation}
where
$$
J = 
\left(
\begin{matrix}
0 & -1 \\
1 & 0 
\end{matrix}
\right)
\qquad
\mbox{and}
\qquad
R = 
\left(
\begin{matrix}
1 & 0 \\
0 & -1 
\end{matrix}
\right) \, .
$$

To understand the dynamic stability of the vortex solution $\psi$, we
will study the spectrum of the operator $A$ as an unbounded operator
on $L^2(\R^2, \R^2)$ with an appropriate domain $D(A)$. 
As is rather well-known, spectral stability of $A$ 
(meaning absence of spectrum in the right half of the complex plane),
need not necessarily imply linear stability 
(in the sense that the zero solution of \eq{linpGPvec} is stable), 
nor nonlinear stability of vortices.  In this paper, we avoid these subtle
issues and confine ourselves to studying spectral stability. After all, the
presence or absence of eigenvalues with positive real part is interesting 
in itself.

The precise definition of the operator $A$ is somewhat involved and
requires the concept of a quadratic form \cite{SR1}. 
Write (only formally for now)
\begin{equation}
L_c = -\frac 12 \laplace + \frac 12 r^2 I\, ,
\qquad
L_w = 2|w|^2 I + |w|^2e^{2m\tht J} R -\mu \, ,
\label{Lcwdef}
\end{equation}
where $I$ is the $2\times 2$ identity matrix. Then 
\begin{equation}
A =  - J (L_c + L_w)\, .
\label{Adec}
\end{equation}
Then define a quadratic map 
$$
q_{Lc}: D(q) = \left(H^1(\R^2, \R^2) \cap L^2(\R^2, \R^2; r^2 )\right)
 \cto \CC
$$
by 
$$
q_{Lc}(\Psi,\Psi) = 
\int_0^{2\pi} \int_0^{\infty} 
\left[ \frac 12|\grad \Psi|^2  + \frac{r^2}2 |\Psi|^2 
\right] \, r\, dr\, d\tht
\, .
$$
The quadratic form $q_{Lc}$ 
is semibounded:
$$
q_{Lc} (\Psi, \Psi) \ge 0 \, .
$$
Note that the space $D(q) = H^1(\R^2, \R^2) \cap L^2(\R^2, \R^2; r^2 )$ 
is dense in $L^2(\R^2, \R^2)$ 
since the Schwartz space is a subset of $D(q)$ and is dense 
in $L^2(\R^2, \R^2)$. 
Here $L^2(\R^2, \R^2; f(r) )$ represents the space of functions 
$\Psi: \R^2 \cto \R^2$ with the bounded norm 
$$
\|\Psi\|^2_{L^2(\R^2, \R^2; f(r) )} = \int_0^{2\pi}\int_0^{\infty} 
|\Psi(r)|^2 f(r) \,r\, dr\, d\tht \, .
$$

The quadratic form $q_{Lc}$ is closed if it has 
a closed graph, i.e., if $D(q)$ is complete under 
the graph norm 
$\|\Psi\|_{+1} = \sqrt{q_{Lc}(\Psi, \Psi) + \|\Psi\|^2_{L^2}}$. 
This is true since both $H^1$ and $L^2(r^2 )$
can be obtained from $L^2$ by completing the space of $C^{\infty}_0$ 
 functions under the $H^1$ and $L^2(r^2 )$ norms 
respectively. Also observe that 
$$
C \left( \|\Psi\|_{H^1}^2 + \|\Psi\|^2_{L^2(r^2 )} \right)  
\ge q_{Lc} (\Psi,\Psi)  + \|\Psi\|^2_{L^2}
\ge c \left( \|\Psi\|_{H^1}^2 + \|\Psi\|^2_{L^2(r^2 )} \right)
$$
for some $C > c > 0$. (The lower bound follows from the semiboudedness; 
the proof of the upper bound is analogous.) 
Then Theorem VIII.15, pp.~278 of~\cite{SR1} yields that
$q_{Lc}$ is the quadratic form of a unique self-adjoint operator $L_c$. 
The domain of the operator $L_c$, denoted by  $D(L_c)$, is dense in $L^2$. 
Clearly 
$H^2(\R^2, \R^2) \cap L^2(\R^2, \R^2; r^2 ) \subset D(L_c)$,
and we have
$$
L_c \Psi = \left(-\frac 12 \laplace + \frac 12 r^2 I 
 \right) \Psi \, ,
$$
for $\Psi \in D(L_c)$ in the sense of distributions. 

It is easy to see that the operator $L_w$ in \eq{Lcwdef}
is bounded on $L^2(\R^2, \R^2)$
Therefore the operator $A = -J(L_c + L_w)$ with the domain 
$D(A) = D(L_c)$ is closed and densely defined in $L^2(\R^2, \R^2)$. 

\subsection{Essential spectrum}
We investigate the spectrum $\sigma(A)$ of the operator $A$ given by \eq{Adec}, regarded
as an unbounded operator on complexified space $D(A) \subset
L^2(\R^2, \CC^2)$.  The spectrum of such an operator in general
consists of two parts: isolated eigenvalues of finite multiplicity
form the discrete spectrum $\sigdisc(A)$, and the remaining part ---
the essential spectrum $\sigess(A)$. The latter is empty as stated in
the next proposition.  

\begin{proposition}
For any $m$ and $\mu$  the 
spectrum of $A$ consists entirely of isolated eigenvalues
of finite multiplicity.  I.e., the essential spectrum 
of the operator $A$ is empty. 
\label{essential}
\end{proposition}

\begin{proof}
The proof has three steps. First, it is easy to see
that the essential spectrum of $L_c$ is empty by 
Theorem XIII.16 on p.~120 of Reed and Simon~\cite{SR}.
Then we prove the same for $JL_c$.  Finally, the generalized Weyl theorem for
non-self adjoint operators yields the same for $JL$. 

Let us prove that the essential spectrum of $JL_c$ is empty. Since
$L_c$ is a non-negative operator, $0$ is not an eigenvalue and therefore
$L_c$ has bounded inverse. Moreover, $L_c$ only has discrete spectrum and its
eigenvalues are isolated with the only possible accumulation point
$\infty$. Then by Theorem XIII.64 on p.~245 of \cite{SR} 
the operator $L_c^{-1}$ is compact.
Now consider the following identity:
\begin{equation}
\la I  - JL_c = (\la L_c^{-1}J^{-1} - I ) JL_c \, .
\label{Lcinv}
\end{equation}
If $\la \notin \sigdisc(JL_c)$, then $\frac 1{\la} \notin
\sigdisc(L_c^{-1}J^{-1})$ and
the right-hand side of \eq{Lcinv} has bounded inverse given by
$$
(\la I  - JL_c)^{-1} = L_c^{-1}J^{-1} (\la L_c^{-1}J^{-1} - I )^{-1} \, .
$$
Here $L_c^{-1}$ is compact, $J^{-1} = -J$ is a bounded operator, $\la
L_c^{-1}J^{-1}- I$ is a compact perturbation of identity, a
Fredholm operator. Moreover, since $\frac{1}{\la} \notin \sigdisc
(L_c^{-1}J^{-1})$, the operator $\la L_c^{-1}J^{-1}- I$ has empty
kernel and is invertible and bounded. Therefore $(\la I - JL_c)^{-1}$
is compact. This implies that $\la I - JL_c$ is invertible with a
compact inverse if $\la$ is not an eigenvalue of $JL_c$. 

It remains to prove that the eigenvalues of $JL_c$ are isolated and of
finite multiplicity, which prohibits discrete spectra to be embedded
in the essential spectrum. 
To show that, consider the resolvent equation 
$$
(I\la - JL_c) u = f \, .
$$
which is equivalent to 
$$
-(I-T(\la)) u = \left(\la L_c^{-1}J^{-1} -  I\right) u = L_c^{-1}J^{-1}
f \, .
$$
The operator $\la L_c^{-1}J^{-1} - I$ is a (multiple of) compact
perturbation of identity and therefore it is also Fredholm.
Also, it is analytic everywhere except for the discrete spectra of $L_c^{-1}J^{-1}$.
By a general result of Gohberg and Krein \cite{KG}, p.21. or 
Kato \cite{Kato}, p.370, the set of values for which $I - T(\la)$ is
not invertible
is at most countable with their only possible accumulation point
infinity.
Therefore the eigenvalues of $JL_c$ are isolated. Also, the
spectral projection on the eigenspace associated with a particular
eigenvalue of $JL_c$ has finite dimensional range, since it is given by
an integral
$$
P_{\la} = \frac{1}{2\pi i} \int_{\Gamma} (\la I - JL_c)^{-1} \, d\la
$$
of a compact operator. Hence the essential spectrum of $JL_c$ is empty
and consist of isolated eigenvalues of finite multiplicity with only
accumulation point infinity, i.e., 
$$
\sigess (JL_c) = \emptyset \, .
$$
 
Finally, we use the generalization of Weyl's theorem to 
non-self-adjoint operators to prove 
$$
\sigess (JL_c) = \sigess (J(L_c + L_w))\, .
$$
It is enough to observe that $J(L_c+L_w)$ is a relatively compact
perturbation of $JL_c$, i.e., that $JL_w(\la I - JL_c)^{-1}$ is compact
whenever $\la \notin \sigdisc(JL_c)$. 
This is true since $JL_w$ is bounded and $(\la I - JL_c)^{-1}$ is compact. 
\end{proof}

Let us point out that although the stability of an axisymmetric vortex in an axisymmetric trap 
does not depend on the trap rotation frequency ${\Om}$, 
it has an influence on the linearized operator $A$. 
The definition of the operator and the proof of emptiness of 
the essential spectra can be adjusted to 
account for the rotation as long as $|{\Om}| \le \om_{tr}$. 
Beyond this threshold it is not clear how to define the operator  and 
whether its essential spectrum stays empty in this parameter regime. 
This threshold may be significant if the axial symmetry is broken. 
A further discussion on other features in this regime can be found in Chapter~12 \cite{KFC}. 
On the other hand, it is easy to see that the eigenvalues (the discrete spectrum) 
of the linearized problem suffer only a shift by a purely imaginary number 
depending on rotation, so stability of this part of the spectrum 
of $A$ is unaffected by  rotation of the trap. 

\subsection{Eigenvalues}
By the result above, the investigation  of the spectrum of the operator $A$
is reduced to study of the eigenvalue equation
\begin{equation}
A\Phi = \la \Phi \, ,
\label{egn}
\end{equation}
which can be rewritten as
\begin{equation}
\left[
\la J + \frac 12 \laplace - \left( \frac 12 r^2 - \mu - 2 |w|^2 \right)
\right] \Phi 
+ |w|^2 e^{2m\theta J} R \Phi = 0 \, .
\label{egn2}
\end{equation}
Similarly as in \cite{PW} 
it is useful to represent solutions of \eq{egn2} in the basis of the
eigenvectors of the matrix $J$:
$$
\Phi =
\Phi_+ 
\left(
\begin{matrix}
1/2 \\
-i/2
\end{matrix}
\right) +
\Phi_-
\left(
\begin{matrix}
1/2  \\
i/2
\end{matrix}
\right)\, .
$$
Since $D(A) \subset L^2(\R^2, \CC^2)$, then $\Phi_{\pm} \in L^2(\R^2, \CC)$. 
Consequently, \eq{egn2}  has the form 
\begin{eqnarray}
\left( 
i \la + \frac 12 \laplace -  \frac12 r^2 +\mu - 2|w|^2  
\right) \Phi_+  -
|w|^2 e^{2im\theta} \Phi_- & = & 0\, , 
\label{bog1} \\
\left(
- i \la + \frac 12 \laplace  - \frac12 r^2 +\mu - 2|w|^2 \right) 
\Phi_-   -
|w|^2 e^{-2im\theta} \Phi_+ & = & 0\, .
\label{bog2}
\end{eqnarray}
Using the information on asymptotic decay and a simple bootstrap argument one can deduce 
that $\Phi_{\pm} \in H_{loc}^{k}$ for each $k > 0$, so 
$\Phi_{\pm} \in C^{\infty} (\R^2, \CC) \cap L^2(\R^2, \CC)$.

Furthermore, decompose $\Phi_{\pm}
(r,\theta)$ at each fixed $r$ into Fourier modes (with shifted indices
for notational ease)
\begin{equation}
\Phi_{\pm} (r, \theta) = 
\sum_{j = -\infty}^{\infty} 
e^{i(j\pm m) \theta}
y_{\pm}^{(j\pm m)} (r) \, .
\end{equation}
After the introduction of the Fourier modes the equation \eq{egn2} 
transforms to an infinite-dimensional system of linear equations.

The system decouples to coupled pairs
for nodes $y_+ = \yp, y_- = \ym$ 
\begin{eqnarray}
\left[
i \la + \frac 12 \laplace_r - \frac{(j+m)^2}{2r^2} 
- \frac 12 r^2 + \mu - 2 |w|^2 \right] y_+
& = & |w|^2y_-  \, ,
\label{Ff1} \\
\left[
- i \la + \frac 12 \laplace_r - \frac{(j-m)^2}{2r^2}  - \frac12 r^2 +  \mu 
- 2 |w|^2 \right] y_- & = &  |w|^2y_+   \, .
\label{Ff2}
\end{eqnarray}
By the symbol $\laplace_r$ we denote the radial Laplace operator 
$\laplace_r = \frac{\del^2}{\del r^2} + \frac 1r \frac
{\del}{\del r}$. 

The proper boundary conditions for this system must be determined only from the system itself and
the property $\Phi_{\pm} \in L^2(\R^2, \CC)$. Here asymptotic behavior of solutions to
\eq{Ff1}--\eq{Ff2} described in Theorem~\ref{abt} of the next Section can be used. It implies 
that the appropriate boundary conditions are
\begin{equation}
\lim_{r\cto 0^+}  y_{\pm}  (r) \quad
\mbox{exists} \qquad \mbox{and} \qquad 
\lim_{r\to \infty} y_{\pm} (r) = 0 \, .
\label{BCfourier}
\end{equation}

Therefore the eigenproblem for $A$ is decomposed into countable many problems
\begin{equation}
L_j y   = i \la R y\, ,
\qquad 
y = \left(y_+, y_- \right)^T  \, ,
\label{egnpair}
\end{equation}
where
$$
L_j = 
\left(
\begin{matrix}
 L_j^+   & 0 \\
 0 & L_j^-  
\end{matrix}
\right)
+ |w|^2
\left( \begin{matrix} 
2 & 1 \\ 1 & 2 
\end{matrix} \right)\, ,
$$
and
\begin{equation}
L_j^{\pm}  =  -\frac12 \laplace_r + \frac {(j\pm m)^2} {2r^2}  +\frac 12 r^2 - \mu\, .
 \label{Ljplus}
\end{equation}
Similarly as before the bootstrap argument gives 
$y \in C^{\infty} \cap L^2(\R^+, \CC^2; r )$. 
The associated inner product is given by 
$$
(y , z) = \int_0^{\infty} \left( 
y_+(r) \overline{z_+(r)} + y_-(r) \overline{z_-(r)} \right) r\, dr \, .
$$
On the other hand, one can argue (as in \cite{PW}) that any solution $(\la, y,j,m, \mu)$ to \eq{egnpair}
defines a solution $(\la, \Phi, m, \mu)$ to \eq{egn}  with  $\Phi \in L^2(\R^2, \CC^2)$.

Note that $L^{\pm}_j = L^{\mp}_{-j}$, a fact connected to the Hamiltonian symmetries
of $A$.  The next proposition analogous to \cite{PW} summarizes the results of this section. 

\begin{proposition}
A complex number $\la$ is an eigenvalue of $A$ 
if and only if for some integer $j$ the system of equations \eq{egnpair} 
have a nontrivial solution satisfying \eq{BCfourier}.
An eigenfunction of $A$ associated with an eigenvalue $\la$ has the form 
$$
v_{\la,j,m}(x,t) =
C e^{\la t} e^{i(j+m)\tht} y_+^{(j+m)}(r) + 
\overline{C e^{\la t} e^{i(j-m)\tht} y_-^{(j-m)}(r)}\, .
$$
If $(y_+,y_-)$ form a solution to \eq{egnpair} for some pair $(\la,j)$,
then $(y_-,y_+)$ form a solution for $(-\la,-j)$ and $(\bar y_-,\bar y_+)$
form a solution for $(\bar\la,-j)$.
\label{intopairs}
\end{proposition}

\subsection{Bounds on unstable eigenvalues}
At a first glance it may seem impossible to solve infinitely many
systems of the form \eq{egnpair}.  Fortunately, similarly as in
\cite{PW} it is possible to restrict the index $j$ for which an
unstable eigenvalue may occur.  
The proof of the following 
Proposition can be found in \cite{GG} but for the sake of completeness and clarity of exposition we 
provide it here as well.

\begin{proposition} All the possible unstable eigenvalues and eigenfunctions of the operator $A$ 
must correspond to bounded solutions of \eq{BCfourier}--\eq{egnpair}  for $j$  satisfying
\begin{equation} 
j \ne 0 \quad \mbox{and} \quad 
|j| < 2m \, .  \label{jindex2} \end{equation}
\label{jtheorem} 
\end{proposition}

\begin{proof}
First we prove that $|j| \le 2m$. The main idea of the proof of this part is the same as in
\cite{GG} but we present it here for clarity. 
Assume the contrary, i.e. $|j| > 2m$, namely, $|j -m| >m$ and $|j+m|>m$. 
The operator $L_j$ is analogous to the operator $L_c$ in Section~\ref{sec:level1-stability} 
and thus the integrals used here are well defined. 
Then for any $y \ne 0$ by integration by parts 
\begin{eqnarray}
(L_j y , y) & = & 
\int_0^\infty 
\left[  
\frac 12 \left( \frac {j+m}{r}\right)^2 |y_+|^2 + 
\frac 12 \left( \frac {j-m}{r}\right)^2 |y_-|^2 
+ \left(\frac 12r^2 + 2|w|^2 -\mu\right) \left( |y_+|^2 + |y_-|^2 \right)
\right. 
\nonumber \\
& & \left. 
+ |w|^2 \left( y_+ \overline{y}_- + \overline{y}_+ y_- \right) 
-\frac 12 \left( \laplace_r y_+  \overline{y}_+ + 
\laplace_r y_- \overline{y}_- \right)
\right]
r\, dr
\nonumber \\ 
& > & 
\int_0^\infty 
\left[
\left( \frac 12 \frac{m^2}{r^2} + \frac 12r^2 + |w|^2 - \mu\right)
\left(|y_+|^2 + |y_-|^2 \right)  
+ \frac 12 \left( |\partial_r  y_+ |^2 + |\partial_r y_-|^2 \right) \right]
r\, dr \, .
\nonumber
\end{eqnarray}
Here we also used the inequality
\begin{equation}
-(y_+ \overline{y}_- + \overline{y}_+ y_-) \le |y_+|^2 + |y_-|^2 \, .
\label{abseq2}
\end{equation}

Hence  $(L_j y , y) >  E_w (y_+) + E_w(y_-)$ where 
\begin{equation}
E_w(f) =  
\frac 12
\int_0^\infty  
\left[ |\partial_r f|^2 + \left(\frac{m^2}{r^2} + r^2 + 2|w|^2 - 2\mu \right) 
|f|^2 \right]  r\, dr \, .
\label{specialenergy}
\end{equation}
The function $w(r)$ is a non-negative minimizer of the linearized energy \eq{specialenergy}
within the family of functions $f(r)$ with $\phi(r,\tht) = e^{im\tht} f(r) \in D(L_j)$. 
Since $E_w(w) = 0$ it follows that $E_w(f) \ge 0$ for all $e^{im\tht}f \in D(L_j)$. 
Hence 
$$
(L_j y , y) > 0 \,  .
$$ 
If $(\lambda,y)$ are an eigenvalue and eigenfunction vector as in \eq{egnpair} 
then also
$$
(i \la R y, y ) = 
i \la (R y, y)  > 0 \, .
$$
Therefore $i\la$  must be real if $|j| > 2m$. 

This result can be also interpreted in terms of Krein signature (see
Section~\ref{s:KS}): the signature of all eigenvalues $\lambda$ for $|j| > 2m$
is positive. 

Next, we prove the stronger result that $|j| < 2m$ and $j \ne 0$. 
First, let us consider the case $j = 0$. In that case one obtains directly 
$$
(L_j y , y)  \ge   E_w (y_+) + E_w(y_-) \, .
$$
The equality is valid only if there is equality in \eq{abseq2}, 
i.e., if $y_+ = -y_-$. Therefore $(L_j y, y)  = 0$ only if $y_+ = -y_-$  
almost everywhere and $E_w(y_+) = 0$. Since $f(r) = w(r)$ is a non-negative solution 
of the differential equation associated with \eq{specialenergy}
$$
-f_{rr} - \frac1r f_r + \frac{m^2}{r^2} f + r^2 f  + 2|w|^2 f - 2\mu f 
= 0 \, ,
$$
the second condition holds only if $y_+ = \alpha w$ 
where $\alpha$ is a complex number, 
and $w = w(r)$ is a real ground state of the energy $E_w(f)$.
Thus 
$$
y = w (\alpha, -\alpha)\, .
$$  
Then the eigenvalue problem \eq{Ff1}--\eq{Ff2} reduces to $i\la w = 0$. 
Hence $\la = 0$ and there are no unstable eigenvalues for $j = 0$. 
Similarly, one can rule out the case $j = \pm 2m$. 
\end{proof}

This result can be also interpreted in terms of Krein signature (see
Section~\ref{s:KS}): the signature of all eigenvalues $\lambda$ for $|j| > 2m$
is positive. 

One can also prove the following estimate which restricts the possible unstable
eigenvalues to lie in a vertical strip.

\begin{proposition}
The real part of every eigenvalue of the operator $A$ is bounded:
\begin{equation}
|\Re \la| < 3 \max\limits_{r>0} |w(r)|^2 < 3(\mu -m) < \infty \, .
\label{realboundth}
\end{equation}
\label{boundtheorem}
\end{proposition}

\begin{proof}
First, recall the simple bootstrap argument justifying that 
$\Phi_{\pm}$ of \eq{bog1}--\eq{bog2} is 
$C^{\infty} (\R^2, \R^2) \cap L^2(\R^2, \R^2)$. 
Then, split the operator $A$ to a vortex-profile-dependent part 
$A_w$ and independent part $A_c$: $A = A_c + A_w$, where
\begin{eqnarray}
A_c & = & J \left[ \frac12 \laplace  - 
\left(\frac12 r^2 - \mu \right)\right] \, ,
\nonumber \\
A_w & =& - J\left[ 2|w(r)|^2 + |w(r)|^2 e^{2m\theta J}R \right]\, .
\nonumber
\end{eqnarray}
Since $J$ and $R$ are constant matrices (bounded by 1 in the matrix norm)
the estimate  \eq{Mest} implies that the norm of the 
profile dependent part $A_w$ is bounded above:
\begin{equation}
\|A_w \|_{L^2} = 
\left\| - 
J\left( 2|w|^2 + |w|^2e^{2m\theta J}R \right) \right\|_{L^2} 
\le 3 M(w)  < \infty \, ,
\label{A1est}
\end{equation}
where $\|\cdot\|_{L^2}$ denotes the operator norm in $L^2(\R^2, \CC^2)$ and 
$M(w) = \max\limits_{r \in (0,\infty)} |w(r)|^2$. 
 
Multiply \eq{egn2} by the smooth complex conjugate $\Phibar$ and integrate 
over $\R^2$  to obtain
\begin{equation}
\la \|\Phi\|^2 =\int_0^{2\pi} \int_{0}^{\infty} A_c \Phi \cdot \Phibar
\, r\, dr\, d\theta 
+\int_0^{2\pi} \int_0^{\infty} A_w \Phi \cdot \Phibar \, r\, dr\, d \theta \, .
\label{phibar}
\end{equation}
The second term on the right hand side of \eq{phibar} can be estimated using  \eq{A1est}  
\begin{equation}
\int_0^{2\pi} \int_0^{\infty} A_w\Phi \cdot \Phibar \, r\, dr\, d \theta \, 
\le 3 M(w) \| \Phi \|^2\, .
\label{A1normest}
\end{equation}
Finally, the real part of 
$\int_0^{2\pi} \int_{0}^{\infty} A_c \Phi \cdot \Phibar \, r\, dr
d\theta$ vanishes, which can be checked by a simple 
but long integration by parts (omitted here).
The statement of the proposition then immediately follows by
Proposition~\ref{profile_behavior}.  
\end{proof}

Propositions~\ref{jtheorem} and \ref{boundtheorem}  restrict the
search for unstable eigenvalues to a finite number of equations (with indices $0 < j < 2m$) 
and to a vertical strip.  
Since one can expect infinitely many stable eigenvalues in both
directions on the imaginary axis it is hopeless to prove that the
imaginary part of an eigenvalue is bounded. 
Nevertheless, the possible number of unstable eigenvalues 
is limited (see Section~\ref{s:KS}).

\section{Evans Function}
While the finite element and the Galerkin approximation methods 
provide a fast and simple way to find eigenvalues of the problem \eq{egnpair}, 
the Evans function technique \cite{CKRTJ, Evans} method has proved to be
the most reliable and robust in certain cases. This approach will be implemented here. It
is parallel to \cite{PW}, where the reader can find many details of the procedure.
 
The main idea of this approach is to identify eigenvalues of the operator $L_j$ as zeros 
of an analytic function $E_j(\la)$. First, write the system \eq{Ff1}-\eq{Ff2} 
as a $4\times 4$ system of first order ordinary differential equations:
\begin{equation}
y' = B(r,j,\la) y
\label{ODEsystem}
\end{equation}
where
\begin{equation}
B = B_{\infty} + B_w \, ,
\quad
y = 
\left(
y_+^{(j+m)} (r), 
\del_r y_-^{(j+m)} (r), 
y_+^{(j-m)} (r), 
\del_r y_-^{(j-m)} (r)
\right)^T
\end{equation}
and
\begin{equation}
B_{\infty} =
\left(
\begin{matrix}
0 & 1 & 0 & 0 \\
k^+ & -1/r & 0 & 0 \\
0 & 0 & 0  & 1 \\
0 & 0 & k^-  &  -1/r
\end{matrix}
\right)\, ,
\qquad
B_w =  |w|^2
\left(
\begin{matrix}
0 & 0 & 0 & 0 \\
4 & 0 & 2 & 0 \\
0 & 0 & 0 & 0 \\
2 & 0 & 4 & 0 
\end{matrix}
\right) \, .
\nonumber
\end{equation}
The coefficients $k^+$ and $k^-$ are given by 
\begin{eqnarray}
k^+(r,\la) & = & \frac{(j+m)^2}{r^2} + r^2 - 2\mu - 2i\la  \, ,
\nonumber \\
k^-(r,\la) & = & \frac{(j-m)^2}{r^2} + r^2 - 2\mu + 2i\la \, .
\nonumber 
\end{eqnarray}
The asymptotic behavior of solutions to \eq{ODEsystem} is described in the next theorem. 
\begin{theorem}
For any $\la \in \CC$ and $m > 0$, $\mu$ real, $j$ integer,
there exist solutions $y_i^{(0)}(r)$ and $y_i^{(\infty)}(r)$, $i = 1, 2, 3, 4$, 
 to the system \eq{ODEsystem} with
the following asymptotic behavior,
$$
\begin{array}{rcll}
y_i^{(0)}(r) & \sim & y_{0i}(r) \qquad
& \mbox{as} \quad {r \cto 0^+}\, , 
\\
y_i^{(\infty)}(r) & \sim & y_{\infty i}(r) \qquad
& \mbox{as} \quad {r \cto \infty}\, .
\end{array}
$$
Here $y_{0i}$ and $y_{\infty i}$, $i = 1, 2, 3, 4$, 
are independent solutions of the asymptotic systems
$$
y_0 = B_0(r,j,\la) y_0\, , \qquad 
y_{\infty} = B_{\infty}(r,j,\la) y_{\infty} \, ,
$$
where
$$
B_0 
= \left( \begin{matrix}
0 & 1 & 0 & 0 \\
l^+ & - \frac 1r & 0 & 0 \\
0 & 0 & 0 & 1 \\
0 & 0 & l^- & - \frac 1r 
\end{matrix} \right)
$$
and
$$
l^+(r) = \frac{(j+m)^2}{r^2}\, , \qquad
l^-(r) = \frac{(j-m)^2}{r^2}\, .
$$
The asymptotic behavior of these solutions as $r \cto + \infty$ is given by 
\begin{equation}
\begin{array}{ll}
y_{\infty 1} \sim 
e^{r^2/2}r^{\alpha_+}
\left(1/r, 1, 0, 0\right)^T \, ,
& y_{\infty 2}  \sim  
e^{r^2/2}r^{\alpha_-}
\left(0, 0, 1/r, 1 \right)^T \, ,
\\
y_{\infty 3} \sim 
e^{-r^2/2}r^{-\alpha_+}
\left(1/r, -1, 0, 0 \right)^T \, ,
& y_{\infty 4} \sim 
e^{-r^2/2}r^{-\alpha_-}
\left(0, 0, 1/r, -1 \right)^T \, ,
\end{array}
\nonumber
\end{equation}
where 
$\alpha_+ = \mu + i\la$, $\alpha_- = \mu - i\la$, 
and as $r \cto 0^+$ by 
\begin{equation}
\begin{array}{ll}
y_{01} \sim 
r^{|j+m|}
\left(1, |j+m|/r, 0, 0 \right)^T \, ,
& y_{02}  \sim 
r^{|j-m|}
\left(0, 0, 1, |j-m|/r  \right)^T \, ,
\\
y_{03} \sim 
r^{-|j+m|}
\left(1, -|j+m|/r, 0, 0 \right)^T \, ,
& y_{04} \sim 
r^{-|j-m|}
\left(0, 0, 1,  -|j-m|/r \right)^T \, .
\end{array}
\nonumber
\end{equation}
\label{abt}
\end{theorem}
The full proof of the theorem which relies on the asymptotic theory of Coppel \cite{Coppel} 
can be found in \cite{thesis}.

The asymptotic analysis reveals that \eq{ODEsystem} has two 
exponentially growing solutions asymptotically equivalent to
$y_1^{(0)}(r)$ and $y_2^{(0)}(r)$ for $r \ll 1$ 
and two exponentially decreasing solutions asymptotically equivalent to 
$y_3^{(\infty)}(r)$ and $y_4^{(\infty)}(r)$ for $r \gg 1$. 
Note that these particular solutions are not in any way  unique. From now on the notation 
$y_1^{(0)}$, $y_2^{(0)}$, $y_3^{(\infty)}$ and $y_4^{(\infty)}$ will always refer to solutions
with the given asymptotics. 
The two-dimensional growing and decaying subspaces non-trivially
intersect only if $\la$ is an eigenvalue.
Their intersection can be detected by vanishing of the Wronskian 
$W(r,\la) = \det (y_1^{(0)}, y_2^{(0)}, y_3^{(\infty)}, y_4^{(\infty)})$. 
By Abel's formula, this determinant satisfies a differential equation  
$W'(r) = \mbox{Tr}\, (B) \, W(r)$.
It is convenient to remove the dependence on $r$ by setting
\begin{equation}
E_j(\la) = -r^2 \det( y_1^{(0)}(r), y_2^{(0)}(r), y_3^{(\infty)}(r), y_4^{(\infty)}(r) )\, .
\label{evansf}
\end{equation}
This Evans function is then independent of $r$.

It is evident that $E_j(\la) = 0$ is a necessary and sufficient
condition for the existence of an eigenvalue.  
A different representation of the Evans function is more convenient
for analytic and numerical use, however, to avoid problems such
as maintaining linear independence of solutions for extreme values of
parameters and  at mode collisions.
It is also desirable to ensure global analyticity of the Evans function
so that the presence and location of eigenvalues may be studied by
means of contour integrals via the argument principle, 

An alternative way to construct and evaluate the Evans function
involves introducing the adjoint system 
\begin{equation}
z ' = - z B (r, j, \la)\, .
\label{adjoint}
\end{equation}
The fundamental matrices $Y(z)$ (its columns are $y_i$) and $Z(z)$
(with rows $z_i$) of systems \eq{ODEsystem} and \eq{adjoint} are
related by $ZY = I$. Therefore Theorem~\ref{abt} (by a simple direct
calculation of the inverse matrix) also guarantees existence of 
four independent solutions of \eq{adjoint} as 
$z_1^{(\infty)}$, $z_2^{(\infty)}$, $z_3^{(\infty)}$ and $z_4^{(\infty)}$ 
such that the matrices $Z^{(\infty)}$ with columns $z_i^{(\infty)}$ 
and $Y^{(\infty)}$ with columns $y_i^{(\infty)}$ satisfy
$Z^{(\infty)} Y^{(\infty)} = I$. 
One can also easily deduce the asymptotic behavior of 
$z_i^{(\infty)}$ as $r \cto \infty$. 
Furthermore, a simple calculation \cite{PW}  shows that 
\begin{equation}
E_j(\la) = \det 
\left(
\begin{matrix}
z_1^{(\infty)} \cdot y_1^{(0)} & z_1^{(\infty)} \cdot y_2^{(0)} \cr
z_2^{(\infty)} \cdot y_1^{(0)} & z_2^{(\infty)} \cdot y_2^{(0)} 
\end{matrix}
\right) \, .
\label{evansfadj}
\end{equation}

Constructing $E_j(\la)$ in this way still involves maintaining the
independence of particular solutions $y_i^{(0)}$ and $z_i^{(\infty)}$, however. 
A quite simple idea to overcome this difficulty 
has been used in a number of earlier works including \cite{PW}.
Instead of considering the system \eq{ODEsystem} 
one can construct a larger $6 \times 6$ system for 
exterior products of solutions to \eq{ODEsystem}. 
This exterior system has a unique 
solution of maximum growth rate given by 
$\hat{y}^{(0)}_{12} = y_1^{(0)} \wedge y_2^{(0)}$ 
and a unique solution of maximal decay rate given by 
$\hat{y}^{(\infty)}_{34}= y_3^{(\infty)} \wedge y_4^{(\infty)}$. 
Corresponding statements hold for the adjoint system.
Evaluation of the Evans function  is then given by the simple formula
\begin{equation}
E_j(\la) =  \hat{z}^{(\infty)}_{34} \cdot \hat{y}^{(0)}_{12}\, .
\label{evansfext}
\end{equation}
It is important to realize that the Evans function given by
\eq{evansfext} is analytic in $\CC$ and it is solution- and
spatially- independent. On the other hand, the values  are the same as
given by \eq{evansf} and \eq{evansfadj}.

These Evans functions have the following symmetries.  The proof 
is the same as in \cite{PW}.

\begin{proposition}
For all integers $j$ and complex numbers $\la \in \CC$,
\begin{itemize}
\item
$\overline{E_j(\la)} = E_j(-\bar\la)$;
\item
$E_j(\la) = E_{-j} (-\la)$.
\end{itemize}
Particularly, $E_j(\la)$ is real for $\la$ purely imaginary and 
$E_0(\la) = E_0(-\la) = \overline{E_0(\overline{\la})}$.
\label{Esymmetries}
\end{proposition}

\section{\label{s:KS}%
Krein Signature} 

In this section we establish some general basic continuation results
on Krein signature for finite systems of eigenvalues in
infinite-dimensional problems, and identify the eigenvalues of
negative signature for the linearized Gross-Pitaevskii equation \eq{GPlin}.

A typical context \cite{GSS, Sandstede} in which Krein signature arises is in 
study of linearized Hamiltonian systems 
\begin{equation}\label{linHam}
v_t = JL v
\end{equation}
on a Hilbert space $X$ with inner product $(\cdot, \cdot)$, 
where $v \in X$, 
$L\colon X \cto X$ is symmetric, $(Lu,v) = (u,Lv)$, and
$J\colon X \cto X$ is invertible and 
skew-symmetric, $(Ju,v) = -(u,Jv)$. 
We are interested in cases when $L$ is unbounded and not a positive operator 
and has a finite number of negative eigenvalues. In particular, we will apply the 
results of this section to the operators $L=L_j$, $J=\hat J = -iR$, see \eq{egnpair}.
Note that individual operators $L_j$ that we consider here do not have all the
symmetries of the Hamiltonian operator $L = L_c + L_w$ of \eq{Lcwdef}.  

To define Krein signature of a discrete eigenvalue $\lambda$ of $JL$ 
(or more generally, a finite collection of such eigenvalues), consider
the restriction of the energy form $(L\cdot, \cdot)$ to the associated
generalized eigenspace. 
If this restriction is positive or negative definite, the Krein signature of the
eigenvalue is positive or negative, respectively. In the case when the
restricted energy is indefinite, the Krein signature of the eigenvalue is
indefinite as well. Note that it is easy to see that the Krein signature of
each eigenvalue off the imaginary axis is zero, since if 
$JLu=\la u$ with $\la\ne-\bar\la$, then $(Lu,u)=0$ since
\begin{equation}
\label{e.JLz}
\la(J\inv u,u)=(Lu,u)=(u,Lu)=(u,\la J\inv u)= -\bar\la(J\inv u,u).
\end{equation}

Let us point out that in \cite{GL,YS} (also see \cite{gril,mackay}) it is
rigorously proved for finite-dimensional systems
that the only way eigenvalues of a system depending on a
parameter can leave the imaginary axis (as a quadruplet, since the symmetry
of the problem forces two complex conjugate pairs of the purely imaginary
eigenvalues to collide simultaneously) is via a collision of two purely
imaginary eigenvalues of opposite signature. 
This behavior is generic -- the
passing of two eigenvalues of opposite Krein signature on the imaginary axis
is an event of codimension one (R.~Koll{\'a}r \& P.~Miller 2010, unpublished), 
as a particular quantity
involving eigenvectors of the associated eigenspaces must vanish.

For simplicity, in what follows we will always assume $J$ is invertible.
Then $\lambda u=JLu$ is equivalent to $Lu=\lambda J\inv u$, and 
thus
\begin{equation}\label{e.LJ}
(Lu, u) = \lambda (J^{-1}u, u)\, .
\end{equation}
For this reason it is sometimes more convenient to use the real
quadratic form $(iJ^{-1}u, u)$ instead of $(Lu,u)$.

\subsection{Finite systems of eigenvalues}

In this subsection we formulate precise statements of some 
key properties of Krein signature that apply to unbounded operators 
of the type we consider.
A key fact that makes the Krein signature so useful for continuation
problems is that for finite systems of imaginary eigenvalues, 
it can never be (positive or negative)
{\em semi-definite} without being {\em definite}.
This is a consequence of the following. 

\begin{lemma}
Let $X$ be a Hilbert space, and let $L$ be a symmetric and $J$ a
skew-adjoint operator on $X$ with a bounded inverse $J^{-1}$. 
Let $\Sigma \subset i\R$ be 
a finite set of discrete purely imaginary eigenvalues of $JL$ 
and let $X_1$ be the corresponding spectral subspace
(the span of all generalized eigenvectors for eigenvalues in $\Sigma$).
Then the quadratic form $(iJ^{-1}u,u)$ is non-degenerate on $X_1$.
\label{GelfLid}
\end{lemma}

\begin{proof}
This result is well-known in the finite-dimensional case
(see \cite[p.~180]{YS})
and the proof here is not very different, 
based on the spectral decomposition 
$$X = X_1 \oplus X_2\, . $$
of the underlying Hilbert space into 
the finite-dimensional $JL$-invariant subspace corresponding to $\Sigma$ 
and its $JL$-invariant spectral complement.
The spectrum of $JL|_{X_1}$ is $\Sigma$, 
and the spectrum of $JL|_{X_2}$ contains no point of $\Sigma$. 
(See Theorem III-6.17 of \cite{Kato}.)

We claim that $(J\inv u,v)=0$ whenever $u\in X_1$ and $v\in X_2$.
This is true because, if $u$ is a generalized eigenvector for some
$\lambda\in\Sigma$, then
$(\la-JL)^nu=0$ for some $n$, and if $v\in X_2$, then $w=(\la-JL)^{-n}
v\in X_2$, and one checks that 
\[
(J\inv u,v)=(J\inv u,(\lambda-JL)^nw)=
(J\inv(\bar\la+JL)^nu,w) = 0.
\]
Now, if the form $(iJ\inv u,u)$ is degenerate on $X_1$, 
it means there exists a non-trivial $u \in X_1$
such that $(J^{-1}u,v)=0$ for all $v \in X_1$, hence for all $v\in X$. 
Hence $J\inv u=0$, so $u = 0$, a contradiction.
\end{proof}

A simple corollary for simple eigenvalues immediately follows.
\begin{corollary}
\label{Simple}
Let $X$, $L$, $J$ be as above. 
Let $\la \in i\R$
be a simple isolated eigenvalue of $JL$ with corresponding
eigenvector $u$.
Then $(iJ^{-1}u, u)$ is non-zero, and if $\la \ne 0$, 
so is the Krein signature $\mysgn(Lu,u)$.
\end{corollary}

In the application we will consider, the operator $L$ depends continuously 
(in a sense we will make precise) on a real parameter $\mu$.  
By the corollary above, the only possibility for
a continuously varying eigenvalue $\la=\la(\mu)$ of $JL(\mu)$ to change its
Krein signature is for it to collide with another eigenvalue, or to cross 
zero.  As a simple consequence of \eq{e.LJ}, 
a simple eigenvalue crossing zero will flip its Krein signature to the opposite.
(If the change is from negative to positive this may correspond to
symmetry-breaking instability \cite{mackay}.)

\begin{corollary}
\label{crossing}
Let $X$ be a Hilbert space, and let
$J$ be a skew-symmetric operator on 
$X$ with bounded inverse. Suppose $s\mapsto L(s)$
is a family of symmetric operators on $X$, 
for $s$ in an interval $\calI\subset\R$. 
Assume that for $s\in\calI$, $(\la(s), u(s))$  
is a continuous family of isolated simple purely
imaginary eigenvalues and corresponding eigenvectors 
for the problem
$$
\la(s) u(s) = JL(s) u(s) \, .
$$
Then $i\lambda(s) (L(s)u(s),u(s))$
has the same sign for all $s\in \calI$ with $\la(s)\ne0$.
\end{corollary}

\begin{proof}
The continuous quantity 
$(iJ\inv u(s), u(s))$ is real and does not vanish for any $s\in\calI$ 
(by Corollary \ref{Simple}).  Then
\[
i\lambda(s)(L(s)u(s), u(s))  
=\la(s)^2 (iJ\inv u(s), u(s)) 
\quad \mbox{for all $s\in\calI$.}
\]
The result follows since $\lambda(s)^2\le0$.
\end{proof}

The previous two results guarantee that the Krein signature of a purely
imaginary eigenvalue does not change unless the eigenvalue crosses the origin
or collides with other eigenvalues. Also, Theorem~\ref{essential} implies
that the operators $JL=iRL_j$ we will consider have only discrete
eigenvalues of finite multiplicity, so only finitely many eigenvalues
may collide at one point, for any value of the parameter $\mu$. 

Such colliding eigenvalues will form a finite system of eigenvalues in
the sense of Kato (see section IV-3.5 of \cite{Kato}), with a
corresponding spectral projection $X\mapsto X_1(\mu)$ that varies
continuously with $\mu$.  If no eigenvalue of negative
or indefinite signature is present in this family before collision,
then the signature is positive definite and must remain so at
collision and after.  Then all eigenvalues must remain on the
imaginary axis after collision --- no pair of eigenvalues can
bifurcate off the axis, because the signature of such a pair is
indefinite.

This is well known for finite-dimensional systems \cite{mackay}.  To
justify these statements about continuation of Krein signature for unbounded
operators, we should first define an appropriate notion of continuity.
The following definition is essentially related to the concept of
convergence of closed operators in the generalized sense of
\cite{Kato}, see  Theorem 2.25, Section IV-2.6.

\begin{definition}
Let $s\mapsto T(s)$
be a family of closed operators in $X$,
defined for $s$ in an interval $\calI\subset\R$. 
We say the family $T(s)$ is {\em resolvent-continuous} on $\calI$ 
if for each $s_0\in\calI$, there exists $\xi\in\CC$ such that
the resolvent map $s\mapsto (\ksi-T(s))\inv$
is defined and norm-continuous for $s$ in a neighborhood of $s_0$. 
\end{definition}

For such a resolvent-continuous family $T(s)$, the resolvent map is
actually (locally) continuous in $s$ for each point $\ksi$ of the
resolvent set of $T(s)$. 

Suppose now that $\Sigma_0$
is a finite set of discrete eigenvalues of $T(s_0)$. This set may be
continued continuously to comprise a finite system of
eigenvalues $\Sigma(s)$ defined on a maximal interval of existence
$\calI'\subset\calI$ in a standard way: Let $\Gamma$ be
any smooth contour (a collection of small circles, say) 
that contains $\Sigma_0$ inside, and no other point of the spectrum of 
$T(s_0)$. The spectral projection
\[
P(s)= \frac{1}{2\pi i}\int_\Gamma (\ksi-T(s))\inv\,d\ksi
\]
is well-defined and continuous in a neighborhood of $s_0$, 
its range $X_1(s)$ is $T(s)$-invariant and has constant dimension, and 
the eigenvalues of the finite-dimensional map $T(s)|_{X_1(s)}$ comprise 
a finite system of eigenvalues that vary continuously with $s$.
The projection $P(s)$ is independent of the choice of $\Gamma$, and
consequently it can be defined continuously in this way for all $s$ in
some {\em maximal} interval $\calI'\subset \calI$. Evidently the
maximal interval must be open.
(At an endpoint of $\calI'$, in general one may have collisions with
other eigenvalues or continuous spectrum, or other pathologies.)

\begin{theorem}
Let $X$ be a Hilbert space and let $J$ be a skew-symmetric operator on 
$X$ with bounded inverse. 
Suppose $s\mapsto L(s)$ is a family of symmetric operators for $s$ in
an interval $\calI\subset\R$ such that the family $JL(s)$ 
is resolvent-continuous.
Suppose $\Sigma(s)$ is a finite system of discrete eigenvalues of
$JL(s)$ as described above, with associated generalized eigenspace
$X_1(s)$, defined for $s$ in some maximal interval $\calI'$.

Then, if the quadratic 
form $(iJ^{-1}u,u)$ is (positive or negative) definite on $X_1(s_0)$
for some $s_0\in\calI'$, it has the same definiteness on $X_1(s)$ for
all $s\in\calI'$.  By consequence, all eigenvalues in $\Sigma(s)$ are 
purely imaginary for all $s\in\calI'$.
\label{rescontthm}
\end{theorem}

\begin{proof} 
The proof is again a rather straightforward extension of 
well-known results from finite dimensions \cite{GL,YS}.
Whenever the form $(iJ\inv u,u)$ is definite on $X_1(s)$ we must
have $\Sigma(s)\subset i\R$ as follows from \eq{e.JLz}.
Considering the positive definite case, let $\rho(s)$ be
the infimum of $(iJ\inv u,u)$ over the unit sphere in $X_1(s)$.
Then $\rho(s_0)>0$ and $\rho$ is continuous on $\calI'$.
If definiteness fails at some point of $\calI'$, then
$\rho(s)>0$ on some maximal strict subinterval of $\calI'$ and
$\rho(s_1)=0$ at an endpoint $s_1\in\calI'$. 
But then $\Sigma(s_1)\subset i\R$ by continuity of the eigenvalues
and so $\rho(s_1)>0$ by Lemma~\ref{GelfLid}, a contradiction.
\end{proof}

Note that in our application the resolvent-continuity condition on
the family of operators is satisfied, as the domain $Y\subset X$ of the 
operators $L(\mu)$ is fixed, and the operators $JL(\mu)$ vary continuously in
$L(Y,X)$.


\subsection{Negative-signature eigenvalues for the GP equation\label{sigK}}
Here we identify all imaginary eigenvalues with negative Krein signature
for $\mu = m+1$, when the linearized problem reduces
to the case of \eq{GPlin}, the Gross-Pitaevskii equation linearized about a trivial
solution. The linearized system \eq{egnpair} for given $j$ and $m$ then reduces to 
$$
\left( H^{j+m} - \mu I \right)  u = i\la u \, ,
\qquad
\left( H^{j-m} - \mu I \right)  v = -i\la v \, ,
$$
where 
$$
H^k = -\frac 12\laplace_r + \frac{k^2}{2r^2} + \frac 12 r^2\, ,
$$
see \eq{egnpair}-\eq{Ljplus}. According to Proposition~\ref{jtheorem} we may assume $j, m > 0$.
Since, by the discussion following \eq{Wrons}, the eigenvalues of $H^{k}$ are
exactly $|k|+1 + 2n$ for $n = 0,1, 2, \dots$, with eigenfunctions $w_n^{(k)}$,
the system eigenvectors $(u,v)^T=(w_n^{(j+m)}, 0)^T$ and $(0, w_n^{(j-m)})^T$ 
correspond to eigenvalues
\begin{equation}
j+m+1 + 2n = m+1 + i\la \, ,
\qquad
|j-m|+1 + 2n = m+1 - i\la \, .
\label{linegv}
\end{equation}
The eigenvalues of the linearized Gross-Pitaevskii equation for $\mu = m+1$ then
come in two families for $j>0$:
\begin{equation}
\la = -i(j + 2n) \, , \quad
\la =
\begin{cases} i(-j + 2n) & \mbox{($0 < j < m$)}\, , \\
 i(-2m + j+2n) & \mbox{($m \le j$)}\, ,
\end{cases}
\end{equation}
for $n=0,1,2,\ldots$.
It is easy to determine the Krein signature of these eigenvalues since 
in the different cases we find
\begin{eqnarray*}
(L(u,v)^{T}, (u,v)^T) & = &((H^{j+m} - \mu I) u , u) = i\la (u,u)  = (j+2n) (u,u)\, , \\
(L(u,v)^{T}, (u,v)^T ) & = & ((H^{j-m} - \mu I) v , v) = -i\la (v,v) = (-j + 2n)(v,v)\, , \\
(L(u,v)^{T}, (u,v)^T ) & = & ((H^{j-m} - \mu I) v , v) = -i\la (v,v) = (-2m + j + 2n)(v,v) \, ,
\end{eqnarray*}
respectively.
Note that according to Proposition~\ref{jtheorem} we can restrict our search for 
possible unstable eigenvalues to the modes with $0 < j < 2m$. 
In particular, this means that when $\mu=m+1$,
the only eigenvalues with negative Krein signature are
\begin{equation*}
\mbox{
\begin{tabular}{llll}
{for $m=1$}: & $\la = -i$  ($j=1$),  &&\\
{for $m=2$}: 
& $\la = -i$  ($j=1$),  
\quad $-2i$  ($j=2$),  
\quad $-i$  ($j=3$).  
\end{tabular}
}
\end{equation*}

\section{Numerical Methods and Results}
%

In this section we first describe in detail the way we evaluate the Evans function and determine the presence and location
of eigenvalues. We then show how the results on Krein signature from Section~\ref{s:KS} are used to justify the finding
of all unstable eigenvalues and interpret eigenvalue collisions. Based on our analytical and numerical results we also 
suggest an improved numerical method to detect unstable eigenvalues by tracking eigenvalues of negative Krein signature. 
Despite we have not used the method in the present work we believe that such an algorithm offers a significant
reduction of computational cost in a wide variety of numerical studies of stability problems that use the Evans function. 
In the further subsections, numerical results are discussed separately for singly- and multi-quantized vortices since the stability diagrams
(diagrams of stable and unstable eigenvalues) reveal different patterns.

\subsection{\label{sec:level2-evans}%
Evans function evaluation}
To attain high precision and stability for all computations, exterior
products were used throughout for numerical evaluation of the Evans function. 
Note that exterior products may be avoided, however, if one uses the
algorithm introduced in \cite{SandComm,zumbrun}.

As easily seen from the asymptotic descriptions in Theorem~\ref{abt},
the behavior of solutions of the system \eq{ODEsystem} is
significantly different for $r \ll 1$ and $r \gg 1$.  
Consequently it is useful to rescale the solution during the integration
process over the interval $(0,\infty)$. 
(The implementation approximates this interval with $[\eps, R]$, where
$\eps = 10^{-7}$ and $R = R(\mu)$ is set in such a way that the vortex
solution is negligible at $r=R$. This $R(\mu)$ is chosen as an
increasing function with $R(0) = 5$, $R(35) = 25$.)  
The aim is to rescale in such a way that the matrix $B(r,\la, j)$
(and the solution) remains appropriately bounded throughout the integration. The
details of the rescaling used here can be found in \cite{thesis}. 

\begin{figure}[t]
\centering
\includegraphics[scale=0.45]{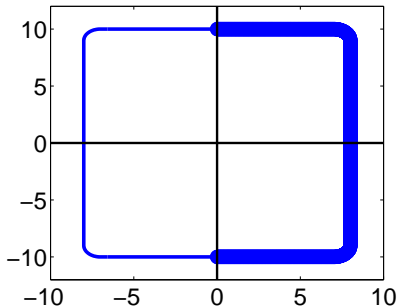}
\caption[The integration contour]{\label{fig:region} 
Contour $\Gamma$ used for counting of eigenvalues. 
Due to the symmetry of the Evans function, 
the computation was restricted to the thick contour in the right-half plane.}
\end{figure}

The presence of unstable eigenvalues is detected by contour
integration using the argument principle, similarly as in \cite{PW},
for $j$'s restricted by Proposition~\ref{jtheorem}.
The algorithm adaptively calculates the argument of the Evans function
$E_j(\la)$ along an
approximately rectangular contour $\Gamma$  which encloses a
bounded region in $\R^2$. The region is pictured on
Fig.~\ref{fig:region}. 
Note that Proposition~\ref{Esymmetries} allows us
to confine the integration to the right half plane (the total
argument change is twice as large) and also reduces the set of
$j$'s for which the calculation is necessary to positive values.
Moreover, Proposition~\ref{boundtheorem} restricts the location of
unstable eigenvalues to a vertical strip. Naturally, it is not
possible to perform numerical calculations without imposing also some
vertical bound for the region enclosed by the contour.  The vertical
bound in the implementation is chosen is such a way that the behavior
of the stable eigenvalues becomes predictable.  We justified that all
the unstable eigenvalues are included {\it a posteriori} by
examining the Krein signature of eigenvalues.

Stable eigenvalues on the imaginary axis are, thanks to the symmetry
of the Evans function, zeros of the real-valued function $E(\la)$.
Hence one can plot that real function and determine the location  and
multiplicity of its zeros within a finite interval. In the actual
implementation this is done automatically --- one first interpolates
the real function by a cubic spline and then uses the Newton method
for locating zeros. In a small neighborhood of a possible double zero,
a very fine mesh was used to resolve any ambiguity. 

The total number of eigenvalues of $E_j(\la)$ enclosed in the region is
then determined by the difference $n_{su} = n_{s+u} -n_{s}$ between
the total number of eigenvalues inside~$\Gamma$, 
$$
n_{s+u} = \frac 1{2\pi i}\oint_{\Gamma} \frac{E_j'(\la)}{E_j(\la)} d \la = 
 \frac 1{2\pi i} \oint_{\Gamma} \mbox{d arg}\, (E_j(\la))  
$$
and the number of (stable) eigenvalues $n_{s}$ on the
imaginary axis, including their algebraic multiplicity.
If $n_{su}$ is not equal to zero it must by
Proposition~\ref{Esymmetries} be an even positive integer and corresponds to
twice the number of pairs of stable and unstable eigenvalues ($\la,
\overline{\la}$) enclosed within $\Gamma$.

The precise location of unstable eigenvalues can be theoretically
determined by the generalized argument principle.
Higher moments 
\begin{equation}
s_k = \sum \la_i^k = 
\frac{1}{2\pi i}
\oint_{\Gamma} \la^k \frac{E_j'(\la)}{E_j(\la)} d \la
\label{firstmoment}
\end{equation}
give the sum of $k$-th powers of positions of all eigenvalues enclosed by
$\Gamma$. 
Given the approximate location of eigenvalues on imaginary axis and number 
of eigenvalues off the axis, the problem 
reduces to solving a set of nonlinear equations. This is particularly efficient 
in the case  $n_{su} = 2$, where only $s_1$ is necessary.  
Unfortunately, even in this case, the numerical error involved can be significant, with a
major contribution coming from a finite difference approximation of $E_j'(\la)$.
Therefore the obtained values are considered only approximate. 
The calculated location is further used to construct a smaller 
contour $\Gamma_s$ which lies solely in the right-half plane and encloses only
a single eigenvalue. The presence of a zero of $E_j(\la)$ inside a smaller
contour is again justified by the argument principle. Its location is
then determined by the generalized argument principle. 
This process can be repeated a few times until a desired precision is attained.
In the implementation the threshold for a precision was set up to be
$10^{-3}$. 

Note that this method of locating eigenvalues does not allow one to
calculate the eigenfunction directly; for that another method must
be used. 

\subsection{Numerical uses of Krein signature} 
Here we show how the Krein
signature  serves to
provide evidence that no unstable eigenvalues are missed in the
numerical computations.
Moreover we suggest a very efficient numerical algorithm based on the
theory in Section~6.  Finally, we discuss an approach that does not
require homotopy and avoids path-following completely. 

In our numerical approach we consider the linearized eigenvalue
problem for a large range of the parameter $\mu$, 
continuously connecting the linearized eigenvalue problem about
a vortex solution to a linearized eigenvalue problem about a trivial
solution. (See \cite{EDCB} for a similar approach).  
The latter problem has only purely imaginary eigenvalues, and their Krein
signatures were determined in Section~\ref{sigK}.
We plot the location of all imaginary eigenvalues in a bounded
interval on the imaginary axis which contains the continuation of all
negative-signature eigenvalues for the whole range of $\mu$ considered
($\mu \in (m+1, 35)$). By calculating the Evans function on the
contour $\Gamma$ and on an appropriate portion of the imaginary axis,
we indirectly track signature changes using the theory of Section~6,
and can thus infer restrictions on possible departures of eigenvalues
from the imaginary axis.  
This tracking process provides strong evidence that there
are no unstable eigenvalues other than the ones we find.

For a significant reduction in computational effort, 
we propose a new numerical method that avoids the computation of
contour integrals for Evans functions.
To identify all unstable eigenvalues, 
it is only necessary to perform the following calculations:
\begin{itemize}
\item
evaluate the Evans function on the imaginary axis close 
to zero to detect any eigenvalues crossing zero;
\item
follow the eigenvalue branches starting at negative Krein signature 
eigenvalues at the reference value of the continuation parameter
(in our case start at  $\mu = m+1$  and follow a total of $2m -1$
different branches),
to detect any collision with positive Krein signature eigenvalues;
\item
follow any branches of eigenvalues that bifurcate into the complex plane.
\end{itemize}
We emphasize that the theory presented in Section~6 justifies that no
unstable eigenvalues can be left out. 
No contour integration is necessary at all since one can just use 
a path-following algorithm (a boundary-value solver). Moreover, 
global analyticity of the Evans function is not needed, since only
zeroes of the real Evans function on the imaginary axis need to be
found initially, and one may use a continuation algorithm \cite{SandComm, SSch}
to track them.

A negative feature of the homotopy technique is that it has
a large overhead if one is interested in only a small set of values of
the parameter.  Therefore we examine ways of avoiding this overhead by
directly evaluating the Krein signature of a given eigenvalue on the
imaginary axis.  This is not feasible with the approach of evaluating
Evans functions using exterior products, since the eigenfunction is
unavailable and it must be calculated separately by a different method.
On the other hand, the restriction of the Evans function 
to the imaginary axis is a real function.
This means that only a continuous representation of the Evans
function is required and problems of rapid growth and numerical dependence are
much less severe. By consequence, it is
not necessary to use exterior products to evaluate the Evans function,
and it can be evaluated either (i) by calculating a
relevant Wronskian of solutions to 
the linear eigenvalue equations  
which makes it easy to compute the eigenfunction,
or (ii) one can use a simple continuation
algorithm proposed by Sandstede \cite{SandComm}. 
By either of these methods, it is possible to evaluate the Krein
signature directly, which then allows one to determine the number of
unstable eigenvalues.

Then for an eigenvalue problem of the form $JL u = \la u$  the following
method may be used.  First, determine (somehow, by using oscillation
theory, for example) the number $n_{total}$ of negative
eigenvalues of $L$. If the number is zero, there are no unstable eigenvalues of
$JL$. If $n_{total} > 0$,  the number gives the total number $n_u$ of pairs of
unstable eigenvalues of $JL$ plus the number $n_s$ of stable eigenvalues with
negative Krein signature \cite{Sandstede,Sandstede2,Pelinovsky}
(also see \cite{GKP}).
If the present method (the Evans
function) detects a number of pairs of eigenvalues off the imaginary axis $n_u
= n_{total}$, there are no other unstable eigenvalues. If $n_u < n_{total}$,
perform a calculation of the Krein signature for eigenvalues on the imaginary
axis. If $n_u + n_s = n_{total}$ there are no other unstable eigenvalues,
otherwise, increase the area of search and repeat the whole process. 

\subsection[Results for singly-quantized vortices $m = 1$]%
{\label{sec:level2-signel}%
Singly-quantized vortices $m = 1$}

Similarily as in \cite{GG, KKC, Pu, Skryabin} we 
found singly-quantized vortex, $m = 1$, to be spectrally stable for all
values of the parameter $\mu$ investigated, $\mu\in (m+1,35)$, corresponding to
number of particles $N \in (0,10^6)$ for the data for \nucl{23}{}{Na} given in
Section~\ref{sec:level1-GP}.  By Propositions~\ref{jtheorem}, \ref{intopairs} 
and \ref{Esymmetries} it suffices to study only $j=1$.
For illustrative purposes the location of the stable eigenvalues ($\mu$ vs. $\Im \la$) 
for $j = 0, 1, 2, 3$, is plotted in Fig.~\ref{m1j0123}.  For the sake of clarity only eigenvalues with
$|\Im \la| < 5$ are shown. 

For small values of $\mu$, close to $\mu_0 = m + 1$, the eigenvalues are close
to the eigenvalues of the reduced uncoupled linear problem neglecting the
$|w|^2$ dependence. 
As $\mu$ increases, certain eigenvalues remain constant: a double eigenvalue
$\la = 0$ and simple eigenvalues $\la = \pm 2i$ for $j = 0$ and simple
eigenvalues $\la = \pm i$ for $j = 1$. These eigenvalues originate in the
symmetries and boosts of the Gross-Pitaevskii equation and are present for
every $m$ (see the Appendix).

\begin{figure}[!tbp]
\centering
\includegraphics[scale=0.5]{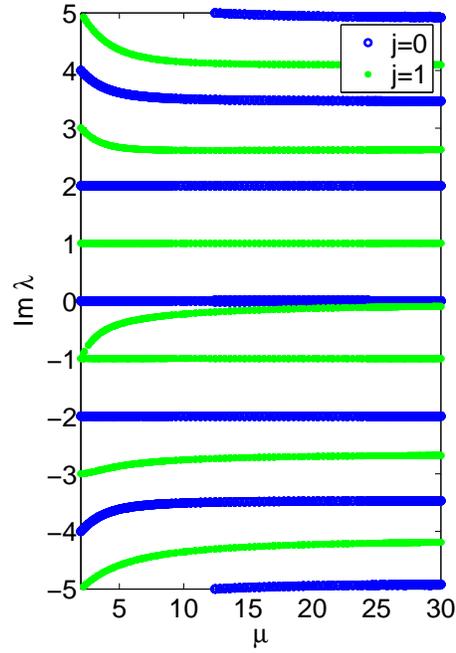}\qquad \qquad
\includegraphics[scale=0.5]{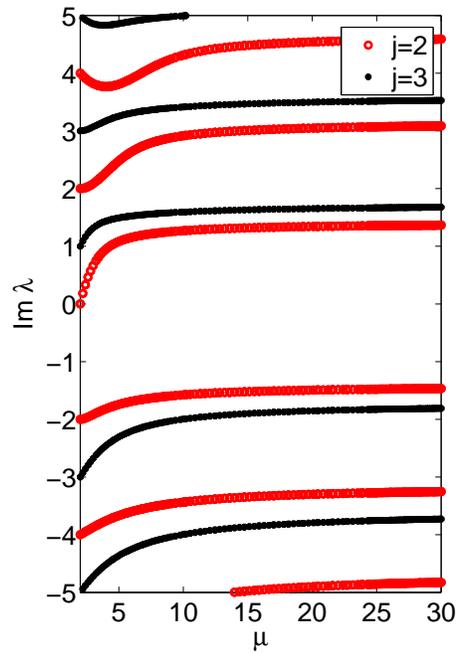}
 \caption[Stable eigenvalues for $m = 1$, $j = 0, 1, 2, 3$, $\mu \in (m+1,30)$] 
{\label{m1j0123}Stable eigenvalues of the linearization about 
a singly-quantized ($m = 1$)
vortex solution, $\mu \in (m+1,30)$:  the eigenvalues corresponding to modes
$j = 0, 1$ (left panel) and  $j = 2, 3$ (right panel).}
\end{figure}

\subsection[Results for multi-quantized vortices $m \ge 2$] %
{\label{sec:level2-multi}%
Multi-quantized vortices $m \ge 2$}

The eigenvalue diagrams for multi-quantized vortices with $m = 2$ show more
complexity.  The radial vortex profile for $\mu\approx 35$ is illustrated in
Fig.~\ref{fig:m2p35}.

In the case $m = 2$ the possible unstable eigenvalues may appear for $|j| = 1,
2, 3$.  The modes $j = 0$  and $j = 1$ demonstrate the same features as in the
case of the singly-quantized vortex with the same constant eigenvalues: a
double eigenvalue $\la = 0$ and simple eigenvalues $\la = \pm 2i$ for $j = 0$
and simple eigenvalues $\la = \pm i$ for $j = 1$ 
(see Fig.~\ref{fig:m2j013} (a)).

A different behavior appears for modes $j = 2$ and $j = 3$, as shown on
Fig.~\ref{fig:m2j013} (b) and Fig.~\ref{fig:m2j2}.  For $j = 3$  there are no
unstable eigenvalues present and the stable eigenvalues do not collide but
rather diverge from each other when they approach each other. 
Collisions do occur for $j = 2$ and cause instability. A collision of two
stable purely imaginary eigenvalues produces a pair of stable and unstable
complex eigenvalues symmetric with respect to the imaginary axis. After an
increase of $\mu$ these two eigenvalues return to the imaginary axis
and split to two purely imaginary eigenvalues as illustrated in
Fig.~\ref{fig:pu}.  This ``collision -- split -- collision'' process (``bubbles
of instability'' \cite{mackay}) repeats regularly for the whole range of $\mu$
studied.

\begin{figure}[!ht]
\centering
\includegraphics[scale=0.46]{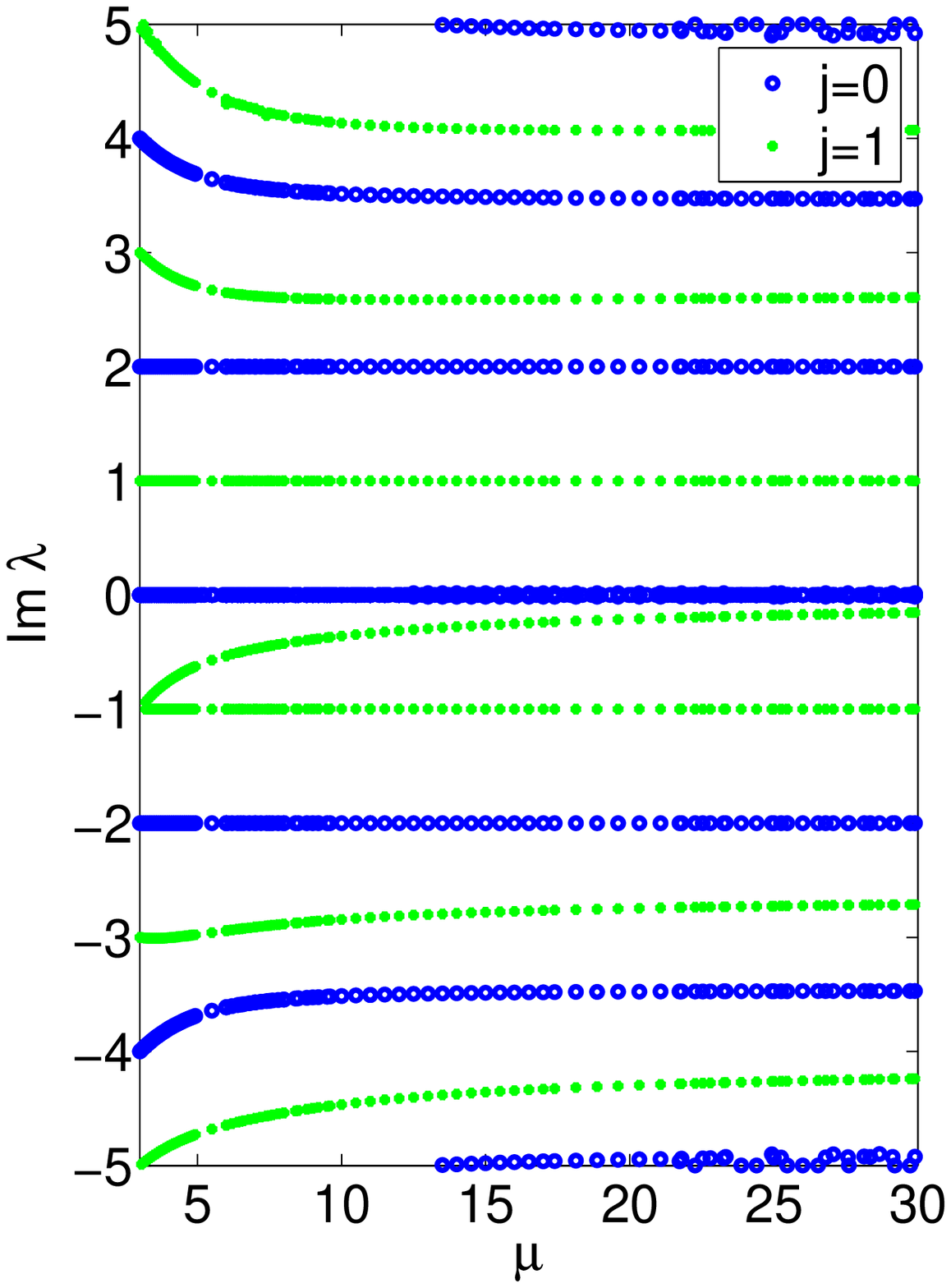}\qquad \qquad
\includegraphics[scale=0.46]{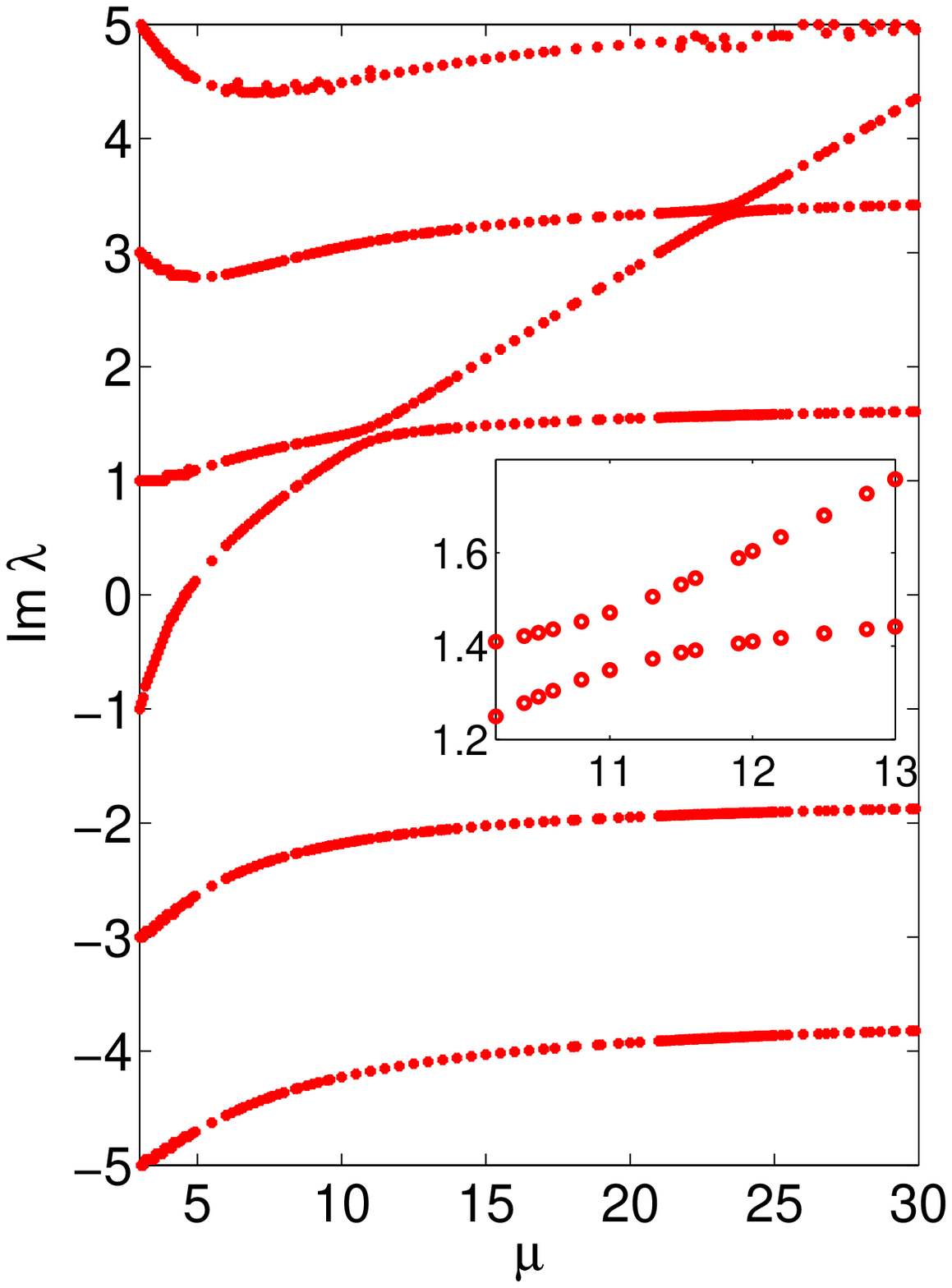}
\caption[Stable eigenvalues for $m = 2$, $j = 0, 1$ and  3, $\mu\in (m+1,35)$] 
{\label{fig:m2j013}Stable eigenvalues 
of the linearization about a multi-quantized
vortex solution, $m=2$, $\mu \in (m+1,35)$:  $j = 0,1$ (left panel) 
and $j = 3$ (right panel), with a detail 
of an avoided eigenvalue collision on the inset.}
\end{figure}

\begin{figure}[t]
\centering
 \includegraphics[scale=0.68]{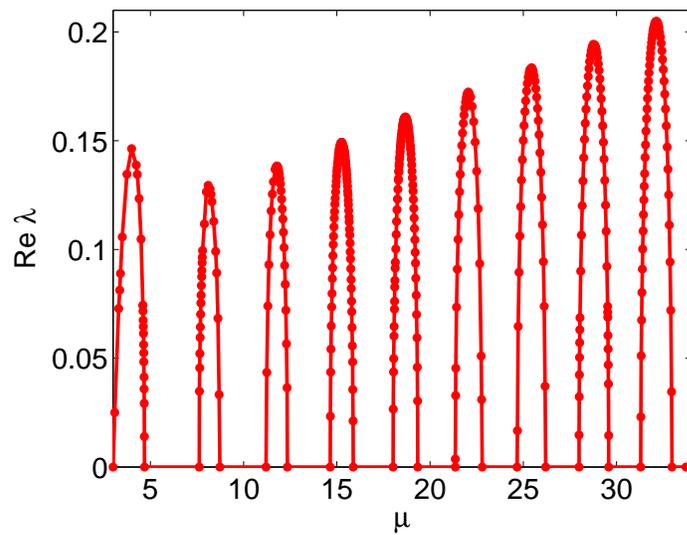}
\vspace{-\baselineskip}
 \caption[Real parts of eigenvalues for $m = 2$, $j = 2$, $\mu \in (m+1,35)$] 
{\label{fig:pu}Stable and unstable 
eigenvalues of the linearization about a multi-quantized
vortex solution, $m=2$, $\mu\in (m+1,35)$, corresponding to the mode $j = 2$.}
\end{figure}

\begin{figure}[!ht]
\centering
 \includegraphics[scale=0.6]{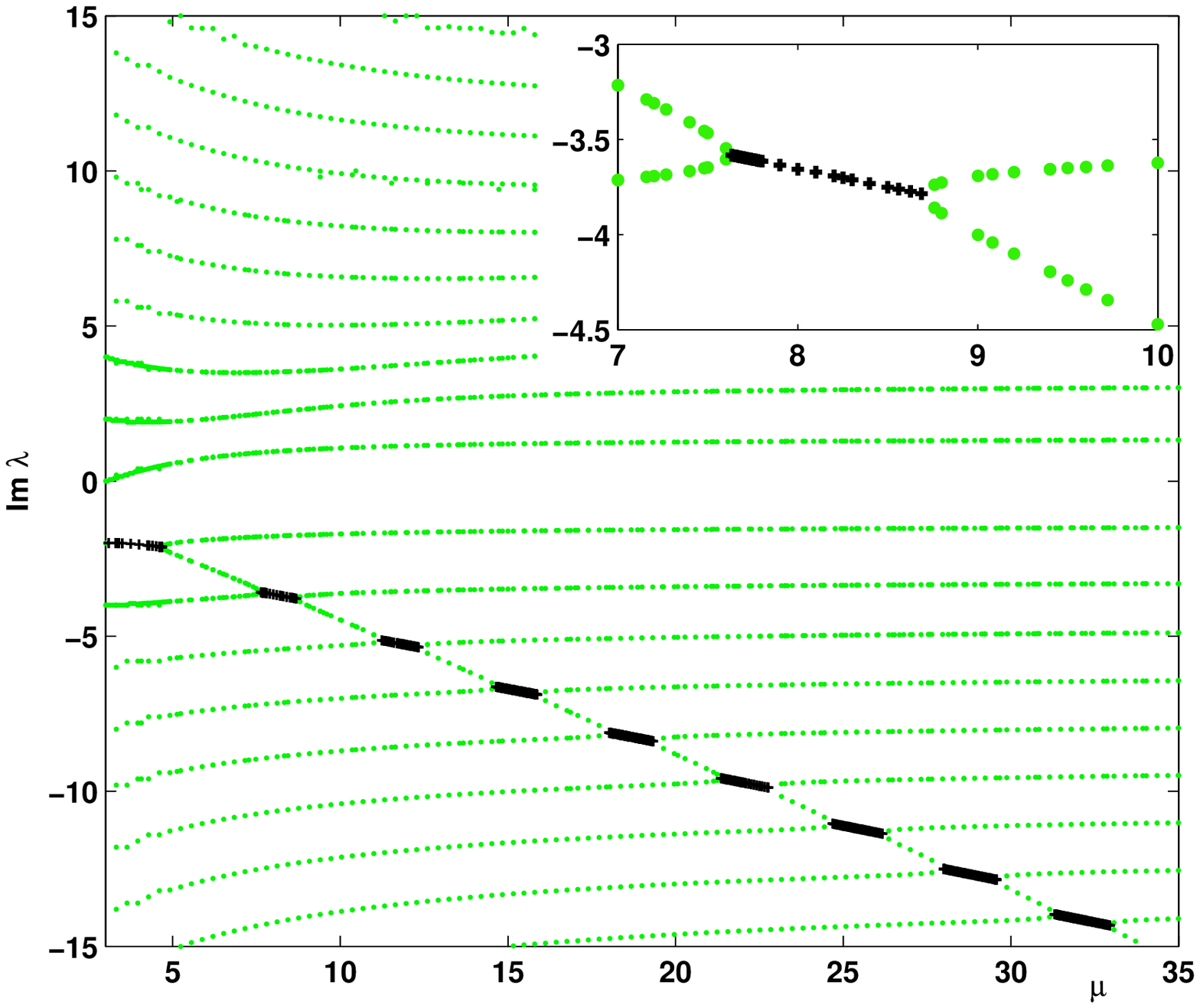}
 \caption[Imaginary parts of eigenvalues for $m = 2$, $j = 2$, $\mu \in (m+1,35)$] 
{\label{fig:m2j2}
Imaginary part of stable (light) and unstable (dark)
eigenvalues of the linearization about a multi-quantized
vortex solution, $m=2$, $\mu\in (m+1,35)$, corresponding to mode $j = 2$. 
A detail of a bubble of instability on the inset.}
\end{figure}

This behavior is caused by the presence of a single eigenvalue with negative
Krein signature \cite{mackay, Skryabin}. The imaginary part of this
eigenvalue is decreasing  with increasing $\mu$, and on its way it encounters
eigenvalues with the opposite signature. After each collision the eigenvalues
split off the imaginary axis and become eigenvalue pairs 
with zero Krein signature symmetric relative to the imaginary axis. 
Reversibility of this process suggests that the eigenvalues come back to the
imaginary axis and the process repeats itself (for a larger parameter~$\mu$).
This indicates a surprising thing:
transitions to instability for larger $\mu$ may happen at a large
frequency ($\mbox{Im}\, \la$ large), and  therefore there is no hope
to confine the imaginary parts of unstable eigenvalues to a finite interval
independent of~$\mu$.  The behavior of the eigenvalues demonstrates strong
agreement with \cite{Pu}. This is also consistent with the results of Seiringer
\cite{S} where he proved that for any $m \ge 1$ for large enough $\mu$ the vortex
becomes energetically unstable in the sense that it cannot be a global
minimizer of the energy and is subject to symmetry breaking.  Also note, that
the maximum real part of unstable eigenvalues grows very slowly with growing
parameter $\mu$.  The maximum of the real part is much smaller than the bound we
were able to obtain in~Proposition~\ref{boundtheorem}.

In the case $j = 3$ the eigenvalues have the same Krein signature and therefore
they cannot split off the imaginary axis The eigenvalue which originates at
$\la = -i$ for $\mu = 3$ has negative signature at first, but after it crosses
zero for  $\mu$ close to 5, it changes its signature to positive according to
Corollary~\ref{crossing}.  Then all eigenvalues have the same signature, making
splitting impossible.  Instead, eigenvalues repel each other upon approach, as
expected by \cite{mackay}.

\subsection[Further Results]%
{\label{sec:level2-furtherresults}%
Further results}

The presence of exponential instability was also checked by direct simulations.
First, an approximate eigenvector was obtained by Galerkin approximation
\cite{GG} and then the Strang-splitting scheme \cite{Bao} was used for time
evolution.  The initial perturbation of a vortex solution by an approximated
eigenvector showed a good agreement with the expected exponential growth.

For the case $m=3$, we performed similar computations but omit details for brevity.
There are more eigenvalues of negative Krein signature, and overlapping 
bubbles of instability. See \cite{Pu} for the analog of Fig.~\ref{fig:pu} 
in this case.

Finally, it would be plausible to describe the asymptotics of the
eigenvalues as $\mu \cto \infty$ when the  condensate approaches the
Thomas-Fermi regime. We were only able to study the asymptotic
behavior of purely imaginary eigenvalues numerically by plotting a
loglog plot of the first order differences of $\la(\mu)$. 
We observed a clear linear trend implying an algebraic approach to a limit. 
Our numerical results agree with the recent analytical results in~\cite{PelKer} for 
stability of vortex solutions in the limiting Thomas-Fermi regime
(see also \cite{GalloPel} for stability of ground states).

The approximate behavior of eigenvalues with a small imaginary part 
for $m = 1$, $j = 1,2$, and $m = 2$, $j = 2,3$, is presented in Table~\ref{T:1}.
For most eigenvalues, the asymptotic behavior as $\mu \gg 1$ is well approximated by 
$$
\la(\mu) = b - \frac{c}{\mu} .
$$ 
On the other hand, certain eigenvalues show different rate of convergence 
clearly distinct from $(-\mu^{-1})$, but we do not have any explanation of this phenomena. 

\renewcommand\arraystretch{1.2}
\begin{table}
\centering
\begin{tabular*}{0.5499\textwidth}{|c|c|l|l|}
\hline
$m=1$   & $j = 1$  & $\la_0 = -3i$ & $-i\la(\mu)= -2.64 - 1.24\, \mu^{-1.02} $ \\
\hline
       &         & $\la_0 = -i$ & $-i\la(\mu) =  0.01 - 1.48\, \mu^{-0.80} $ \\
\hline
       &         & $\la_0 =  3i$ & $-i\la(\mu) =  2.73 - 0.17\, \mu^{-0.14} $ \\
\hline
      & $j=2$    & $\la_0 = -2i$ & $-i\la(\mu) = -1.41 - 1.82\, \mu^{-1.04} $ \\
\hline
      &          & $\la_0 =  0$ & $-i\la(\mu) =  1.41 - 1.61\, \mu^{-1.04} $ \\
\hline
      &          & $\la_0 =  2i$ & $-i\la(\mu)=  3.16 - 2.43\, \mu^{-1.02} $ \\
\hline
$m = 2$ & $j = 2$ & $\la_0 = -4i $ & $-i\la(\mu) = -1.41 - 2.69\, \mu^{-0.96} $ \\
\hline
      &          & $\la_0 =  0 $ & $-i\la(\mu) =  1.38 - 7.96\, \mu^{-1.39} $ \\
\hline
      &          & $\la_0 =  2i $ & $-i\la(\mu) =  3.11 -27.66\, \mu^{-1.59} $ \\
\hline
      & $j = 3$  & $\la_0 = -3i $ & $-i\la(\mu) = -1.74 - 5.14\, \mu^{-1.06} $ \\
\hline
      &          & $\la_0 = -i $ & $-i\la(\mu) =  1.73 - 3.97\, \mu^{-1.02} $ \\
\hline
      &          & $\la_0 =  i $ & $-i\la(\mu) =  3.61 - 5.30\, \mu^{-0.97} $ \\
\hline
\end{tabular*}
\caption{Approximate asymptotic behavior of eigenvalues for $m = 1, 2$.}
\label{T:1}
\end{table}

\section*{Acknowledgments}
This work was supported by the
National Science Foundation [DMS07-05563 to R.K., DMS06-04420,  DMS09-05723 to R.L.P.)]; 
the Center for Nonlinear Analysis (CNA) under National Science Foundation Grant [0405343, 0635983];
the European Commmission Marie Curie International Reintegration Grant [PIRG-GA-2008-239429];
a Dissertation Fellowship from the University of Maryland;
and a Rackham Fellowship from the University of Michigan, Ann Arbor.

\appendix
\section*{Appendix: Symmetries and eigenvalues}
The symmetries of the Gross-Pitaevskii equation and its linearization 
imply the presence of a special set of eigenvalues. 

{\em Phase symmetries.} For any $m$,
$$
\la= 0 \quad \mbox{(for $j = 0$)}
$$ 
is a constant double eigenvalue for all $\mu \ge  \mu_0$. 
Its multiplicity comes from the symmetry of the Gross-Pitaevskii equation under
a phase change (this generates an eigenvector) and under a change of
standing-wave frequency (this generates a generalized eigenvector).  
A detailed discussion on these two symmetries and their implications for spectra
is given in \cite{PW}. 

{\em GGV boost.}
Due to the presence of the harmonic potential, the other usual invariants of the nonlinear
Schr{\"o}dinger equation---spatial translations---do not apply. 
Similarly, one cannot perform the typical Galilean boost.
Instead, one can apply a
boost for quadratic potentials that was described by Garc{\'\i}a-Ripoll {\it et al.}
\cite{GGV}.
This particular boost transforms any solution 
$\Psi(\rr,t)$  ($\rr = (x,y)$)
to a new solution of the form
\begin{equation}\label{GGV}
\Psi_R (\rr, t) = \Psi(\rr - R(t),t) \exp(i\thtt(\rr,t))\, ,
\end{equation}
where $R(t)$ is the path of a classical particle moving in the potential well.
For the harmonic potential $V(\rr)=\frac12 \rr^2$, this requires simply
$R_{tt} = -R$, so that each component of $R(t)$ can be any linear combination of 
$\cos t$ and $\sin t$. According to \cite{GGV}, $\thtt$ is given up to
a constant by $\thtt(\rr,t)=\rr\cdot R_t$.

As discussed in \cite{PW}, any one-parameter family $u_{\tau}$ of solutions to
\eq{GPfinal} gives rise to a solution of the corresponding equation linearized
about $u_0$, given by 
$$
\tilde{u} := \partial_{\tau} u_{\tau} |_{\tau = 0} \, .
$$
In the case of the GGV boost \eq{GGV}, setting $R(t) = \tau(\cos t, \sin t)^{T}$ 
gives $\thtt(r,\tht,t) = \tau r \cos(\tht - t)$.  
Then $u_{\tau}$ is given by \eq{GGV} with
$\Psi(\rr,t)=\psi(t,r,\tht)$ as given by \eq{vortexS} with $w(r)$ satisfying \eq{GPradial}.
Hence 
$$
\tilde{u} =  \binom{\cos t}{\sin t} \cdot \grad \psi(t,r,\tht)    + i r \cos(\tht - t) \psi(t,r,\tht) \, ,
$$
which can be also written as
$$
\tilde{u} = e^{-i\mu t} e^{im\tht} 
\left[w'(r) \cos(\tht - t) + i \frac{mw}{r} \sin(t-\tht) + irw\cos(t-\tht) \right] \, .
$$
The vector $\tilde{\Psi} = (\Re \tilde{u}, \Im \tilde{u})^{T}$ 
then satisfies the linearized equations \eq{linpGPvec} and has the form
\begin{eqnarray*}
e^{\mu t} \tilde{\Psi} & = &
\binom{\cos m\tht}{\sin m\tht} \binom{\cos \tht}{\sin \tht} \cdot \binom{\cos t} {\sin t} w'(r) \\
&& +  \binom{-\sin m\tht}{\cos m\tht}\!\left[- \binom{\cos \tht}{\sin \tht}\! 
\cdot \! \binom{-\sin t} {\cos t} \frac{mw(r)}{r}
+ \binom{\cos \tht}{\sin \tht} \!\cdot\! \binom{\cos t} {\sin t} rw(r) \right].
\end{eqnarray*}
In the decoupling to Fourier modes this solution yields a solution to the system \eq{Ff1}--\eq{Ff2} for 
$j = 1$, $\la = i$ 
and
$$
y_+ = \frac 12 \left[w'(r) - \frac{mw(r)}{r} + rw(r)\right]\, ,\quad
y_- = \frac 12 \left[w'(r) + \frac{mw(r)}{r} - rw(r)\right]\, .
$$
The similar solution can be also obtained for $\la = - i$: 
$$
y_+ = \frac 12 \left[w'(r) - \frac{mw(r)}{r} - rw(r)\right]\, ,\quad
y_- = \frac 12 \left[w'(r) + \frac{mw(r)}{r} + rw(r)\right]\, .
$$
Therefore the GGV boost implies the existence of the eigenvalues  (for any $m$)
$$
\la = \pm i \quad\mbox{(for $j=1$).}
$$ 

{\em Breather boost.}
Finally, we describe the source of the presence of the eigenvalues
$$
\la = \pm 2i  \quad\mbox{(for $j=0$)},
$$ 
which appear for every $m \ge 1$ and $\mu \ge \mu_0$.
These eigenvalues corresponds to a ``breathing'' symmetry of the Gross-Pitaevskii equation
in $2+1$ dimensions, as explained by Pitaevskii and Rosch in \cite{PR}.
Here we provide a somewhat different perspective, showing that
the symmetry corresponds to the Talanov lens
transformation via an exact transformation to the cubic Schr{\"o}dinger equation.
Rybin {\it et al.} \cite{RB} noted that the radially symmetric \gp\ equation in
2+1 dimensions transforms exactly to the cubic \schr\ equation
without potential, according to a transformation described originally 
by Niederer \cite{Ni74} for the linear \schr\ equation.  
More generally, they noted that this transformation works in any
dimension as long as the nonlinearity is critical, of the form
$|u|^{4/n}u$ in $n$ space dimensions. This symmetry was used by Carles
\cite{Carles} to study various mathematical aspects of the \gp\
equation.

By transforming from GP to cubic \schr, using a Talanov
lens transformation \cite{Tal}, and transforming back, one can get a
self-transformation of GP (see also \cite{Ghosh}).  
The symmetry that we will describe is somewhat more general, and
works as follows. Suppose that $v(x,t)$ is any solution
of the equation
\begin{equation}\label{gp1}
\rmi \partial_t  v + \frac{1}{2}\Delta v -\frac{1}{2}{\om^2} |x|^2 v - \lambda |v|^{p} v = 0
\end{equation}
where $t\in\R$ and $x\in\R^n$. Let 
\begin{equation}\label{udeff}
u(t,x) = a \rme^{-\rmi b |x|^2/2} v(c x,\tau)
\end{equation}
where $a$, $ b$, $c$ and $\tau$ are functions of time $t$. 
Then we find that 
\begin{equation}\label{ueq}
\rmi \partial_t u  +\frac{1}{2}\Delta u - \frac{1}{2} {\nu^2} |x|^2 u - \gamma |u|^p u =0 ,
\end{equation}
provided that $b$ and $c$ satisfy
\begin{equation}\label{bcsys}
b'- b^2 +{\om^2} c^{4}={\nu^2} , \quad c'= b c,
\end{equation}
and 
\begin{equation}\label{atg}
a=c^{n/2}, \quad \tau'=c^2, \quad \gamma= \lambda c^{2-np/2}. 
\end{equation}
Of course, $b=0$, $c=1$,  always works with ${\nu^2}={\om^2}$.

The symmetry is most interesting when the nonlinearity is critical, 
meaning $p=4/n$. Then $\gamma=\lambda$.  In 2+1 dimensions this 
corresponds to the cubic nonlinearity.
For ${\om^2}={\nu^2}=0$ we recover the well-known 
Talanov lens transformation for the critical \schr\ equation:
\begin{equation}\label{talanov}
b= \frac{b_0}{1-b_0t}, \quad c = \frac{c_0}{1-b_0 t},
\quad \tau = \frac{c_0^2 t}{1-b_0t}.
\end{equation}
With ${\om^2}=0$ and constant ${\nu^2}>0$ we get the transformation 
from the critical
\schr\ equation to GP as described by Rybin {\it et al.} \cite{RB} and Carles
\cite{Carles}:
\begin{equation}\label{nls2gp}
b = \nu\tan \nu t, \quad c = \frac{1}{\cos \nu t},
\quad \tau = \frac{\tan \nu t}{\nu}.
\end{equation}
In the other direction with constant ${\om^2}>0$ and ${\nu^2}=0$ we
get \cite{Carles}:
\begin{equation}\label{gp2nls}
b= -\frac{{\om^2} t}{1+{\om^2} t^2},
\quad
c= \frac{1}{\sqrt{1+{\om^2} t^2}},
\quad
\tau= \frac{\arctan \omega t}{\omega}.
\end{equation}

The consequences of this transformation for the \gp\ equation in 2+1 dimensions 
are striking. From any solution of \eq{gp1}, one gets from \eq{udeff}
a solution obtained by a {\it breather boost} --- a time-periodic
dilation of space with an appropriate radial phase adjustment. 
See \cite{PR} for more discussion and a relation to representations
of the group SO(2,1).

A short calculation shows that this transformation corresponds to 
the eigenvalues $\pm 2i$ of the normalized linearized equations \eq{Ff1}--\eq{Ff2}.
The relation to these eigenvalues can be seen also from a simple consideration, that
this self-transformation  produces small oscillations at exactly {\it twice} the trap
frequency for classical oscillations in the harmonic potential
$\frac12{\om^2}|x|^2$.
The exact formulae for this breathing boost also show that breathing
oscillations need not be small. Also note that a breather boost can be applied to any
solution, not only standing waves.

\end{document}